\documentclass[11pt]{amsart}

\usepackage{amsmath, amsfonts, amssymb,amsthm}
\usepackage{amstext}
\usepackage{mathrsfs}

\usepackage{float,epsf,subfigure}
\usepackage[all,cmtip]{xy}
\usepackage{hyperref}
\usepackage{algorithm}
\usepackage{algorithmic}
\usepackage{mathtools}
\usepackage{float}
\usepackage{bm}
\theoremstyle{plain}
\newtheorem{theorem}{Theorem}[section]

\newtheorem{cor}[theorem]{Corollary}
\newtheorem{conjecture}[theorem]{Conjecture}
\newtheorem{def-thm}[theorem]{Definition-Theorem}
\newtheorem{lemma}[theorem]{Lemma}

\newtheorem{defi}[theorem]{Definition}

\newtheorem*{tha}{Theorem A}
\newtheorem*{thb}{Theorem B}
\newtheorem*{thc}{Theorem C}
\newtheorem*{thd}{Theorem D}

\theoremstyle{definition}

\newtheorem{remark}[theorem]{Remark}

\def\min{\mathop{\mathrm{min}}}

\begin{document}
\title[Holomorphic curves in moduli spaces]{Holomorphic curves in moduli spaces of  polarized Abelian varieties}
\author[X.-J. Dong]
{Xianjing Dong}

\address{School of Mathematics \\ China University of Mining and Technology \\ Xuzhou, 221116, P. R. China}
\email{xjdong@amss.ac.cn; xjdong05@126.com}


\date{}
\maketitle \thispagestyle{empty} \setcounter{page}{1}

\begin{abstract}  We study the value distribution of holomorphic curves from a general 
open Riemann surface  into a smooth logarithmic pair $(X, D).$ 
By stochastic calculus, we first 
 obtain   a version of tautological inequality (proposed by McQuillan) 
 and a logarithmic derivative lemma. Then, one
  uses them to establish a 
    Second Main Theorem of Nevanlinna theory for  pair $(X, D)$ under certain conditions. 
Finally, we apply  the Second Main Theorem 
   to study  the
   holomorphic curves  from a general open Riemann surface into   
  certain special  moduli spaces of  
  polarized Abelian varieties intersecting  boundary divisors.
\end{abstract}

\vskip\baselineskip

\setlength\arraycolsep{2pt}
\medskip

\section{Introduction}
\subsection{Main results}~

To begin with, we  shall  review a conjecture of Vojta  in Nevanlinna theory.
Let $(X,D)$ be a smooth logarithmic pair   over $\mathbb C,$ i.e.,  $X$ is a smooth complex projective variety and $D$ is a 
normal crossing divisor on $X.$ Denote $K_X(D)=K_X\otimes\mathscr O_X(D),$ where $K_X$ is the canonical line bundle over $X$ and 
$\mathscr O_X(D)$ is the holomorphic line bundle defined by $D.$
Let us consider the finite ramified covering 
 $\pi: B\rightarrow \mathbb C,$  where $B$ is an open (connected) Riemann surface.
 Given an  ample  line bundle $A$ over $X,$
Vojta conjectured the following Second Main Theorem (\cite{vojta}, Conjecture 27.5) that 
\begin{conjecture}[Vojta, \cite{vojta}] \label{conj} \emph{For any holomorphic curve $f: B\rightarrow X$ 
whose image is not contained in ${\rm{Supp}}D,$
 we have}
\begin{eqnarray*}
T_{f,K_X(D)}(r) &\leq_{\rm exc}&
N^{[1]}_f(r,D)+N_{\rm{Ram}(\pi)}(r)+
O\Big(\log T_{f, A}(r)+\log r\Big),
\end{eqnarray*}
where $\leq_{\rm exc}$ means that an inequality holds for  $r>1$ 
outside a subset of finite Lebesgue measure. 
\end{conjecture}
 
Recently, Sun \cite{sun} considered these pairs $(X, D)$  which can be interpreted as  smooth compactification of the base space of a family, i.e.,  the pair $(X, D)$ whose complement $U= X\setminus D$ carries  a family of smooth polarized varieties. 
Suppose  that $(\psi: V\rightarrow U)$ is a smooth family of  polarized smooth varieties with semi-ample canonical sheaves and fixed  Hilbert polynomial $h,$ such that the induced classifying mapping from $U$ to   moduli scheme $\mathcal M_h$ 
is quasi-finite. 
In Section 5 (\cite{sun}, Section 2), we will review that there exists   an ample line bundle $A$ over the base space $X,$ which is closely related to the direct image sheaf of the family $\psi.$
Consider  a holomorphic curve $f: B\rightarrow X$ 
whose image is not contained in ${\rm{Supp}}D,$ Sun 
 showed   the following Second Main Theorem  
\begin{eqnarray*}
T_{f, A}(r) &\leq_{\rm exc}&
\frac{d+1}{2}\Big(N^{[1]}_f(r,D)+N_{\rm{Ram}(\pi)}(r)\Big)+
O\Big(\log T_{f, A}(r)+\log r\Big),
\end{eqnarray*}
where $d$ is  the  fiber dimension of the  family  $\psi.$
Furthermore,  Sun 
applied  his Second Main Theorem  to study the holomorphic curves into certain modular varieties (\cite{sun}, Theorem D).
 For more details about  moduli spaces of smooth polarized varieties, we refer the reader to \cite{popp, vz4}. 

In this paper, we will revisit Vojta's conjecture and Sun's results in a  very different way.
Instead of  $B,$ we study  the value distribution of a holomorphic curve 
$f: S\rightarrow X,$
 from  a geometric point of view, where $S$ is a general open Riemann surface.
Instead of  $N_{\rm{Ram}(\pi)}(r),$
we wish to give a quantitative  term  depending only on geometric nature  of $S.$  As two most  important cases,  we also wish to include the classical results for $\mathbb C$ and the unit disc.

 The main strategy used in the paper is the probabilistic approach, namely, the technique of Brownian motions. 
 Applications of  Brownian motion theory in Nevanlinna theory  can be traced  back to  1986,  Carne \cite{carne} 
 re-formulated the Nevanlinna's functions of meromorphic functions on $\mathbb C$ in terms of Brownian motions.
 Via  the stochastic calculus, Carne provided  a probabilistic proof of  Nevanlinna's two fundamental theorems \cite{nev}, i.e.,  First Main Theorem and  Second Main Theorem. 
    Later, Atsuji  wrote a series of papers in developing this  technique. 
   One of the most important work  of  Atsuji  on  the Nevanlinna theory may be the establishment of  
    a  Second Main Theorem of meromorphic functions  on a 
 complete K\"ahler manifold of non-positive sectional curvature (see \cite{at, atsuji}). 
Via the stochastic calculus of  Brownian motions, 
 Dong-He-Ru \cite{dong} 
  also  gave  a  probabilistic proof of H. Cartan's theory 
   for holomorphic curves into a complex projective space intersecting hyperplanes in general position.  In this paper, we  shall apply  Brownian motion more technically  
  to  the study of  value distribution  of holomorphic curves in moduli spaces. 
Although the stochastic  calculus of  Brownian motions such as Coarea formula, It\^o formula and Dynkin formula is similarly used  as Carne and Atsuji,  
the  technical route   differs  from what it used to be.  
Atsuji  focused  his  attention  on  meromorphic functions on manifolds of higher dimension, 
however,  
what we are concerned with  are  holomorphic curves into a complex projective variety. 
They are  two  different research branches  in Nevanlinna theory
and  have their own research approaches.  
A  quite different technical route  in  the  paper 
is the Logarithmic 
Derivative Lemma 
(Theorem \ref{ldl2}), 
 that  has not been crossed by Atsuji. 
  This work, nevertheless, 
  is also very benefited from   the contributions of Atsuji  to the estimation 
of Green functions for geodesic balls  (Lemma \ref{zz}).

Let us introduce the main theorems of this  paper. Some notations will be provided later.
By uniformization theorem,  we can equip $S$ with a complete Hermitian metric written as $ds^2=2gdzd\bar{z}$  in a local holomorphic coordinate $z,$
such that its
   Gauss curvature
$K_S\leq0$  associated to the metric $g,$  here $K_S$ is defined by
$$K_S=-\frac{1}{2}\Delta_S\log g=-\frac{1}{g}\frac{\partial^2\log g}{\partial z\partial\bar z},$$
where $\Delta_S$ is the Laplace-Beltrami operator on $S$ associated to the metric $g.$
Evidently, $S$ is a complete K\"ahler manifold with  associated K\"ahler form
$$\alpha=g\frac{\sqrt{-1}}{\pi}dz\wedge d\bar{z}$$
in a local holomorphic coordinate $z.$
Now, fix $o\in S$ as a reference point.  We denote by $D(r)$ the geodesic ball centered at $o$ with radius $r,$ and by $\partial D(r)$ the boundary of $D(r).$
By Sard's theorem, $\partial D(r)$ is a submanifold of $S$ for almost all $r>0.$
Set
\begin{equation}\label{kappa}
  \kappa(t)=\min\big\{K_S(x): x\in \overline{D(t)}\big\}.
\end{equation}
Then $\kappa$ is a non-positive, decreasing  continuous function  on $[0,\infty).$

We establish a  Second Main Theorem for pair $(X, D)$  as follows
\begin{tha}[=Theorem \ref{main}] Let $(X, D)$ be a smooth logarithmic pair over $\mathbb C$ with $U= X \setminus D.$ Assume  that there is   a smooth family $(\psi: V\rightarrow U)$ of polarized smooth varieties with semi-ample canonical sheaves and a given Hilbert polynomial $h,$ such that the induced classifying mapping from $U$ into  moduli scheme $\mathcal M_h$ is quasi-finite. 
Then for any holomorphic curve  $f:S\rightarrow X$ whose image is not contained in ${\rm{Supp}}D,$
we have
\begin{eqnarray*}
T_{f,A}(r) &\leq_{\rm exc}&
\frac{d+1}{2}N^{[1]}_f(r,D)+
O\Big(\log T_{f, A}(r)-\kappa(r)r^2+\log^+\log r\Big),
\end{eqnarray*}
where $A$ is an ample line bundle  over $X$ close to $\psi$ given as above, and  $d$ is the fiber dimension of the family $\psi.$  More precisely, if $S$ is the Poincar\'e disc $($take $o$  as the center of disc$)$,  then
\begin{eqnarray*}
T_{f,A}(r) &\leq_{\rm exc}&
\frac{d+1}{2}N^{[1]}_f(r,D)+
O\Big(\log T_{f, A}(r)+ r\Big).
\end{eqnarray*}
\end{tha} 

We   find in  Theorem A that $N_{\rm{Ram}(\pi)}(r)$ is removed and  a new term $-\kappa(r)r^2$  is appeared, which depends on the curvature of $S,$  and  is  more intuitive than $N_{\rm{Ram}(\pi)}(r).$  In  case  $S=\mathbb C$ (equipped with  standard  Euclidean metric), one has       
$\kappa(r)\equiv0.$  By  Remark \ref{re12},  $T_{f,A}(r)$ (characteristic function) and $N^{[1]}_f(r,D)$ (truncated counting function) agree with the classical ones.  
The case $S=\mathbb D$ (unit  open disc equipped with Poincar\'e metric)  will be discussed in Section 1.2 below.

\begin{remark}   In this paper, the Poincar\'e disc  is the unit disc $\mathbb D=\{z\in\mathbb C: |z|<1\}$ equipped with    Poincar\'e metric 
 $$ds^2=\frac{4dzd\bar z}{(1-|z|^2)^2},$$
 which is a complete hyperbolic metric of Gauss curvature $-1.$
 \end{remark}
 \begin{thb}[=Corollary \ref{cor3}]  Assume the same conditions as in Theorem {\rm A}.
Let $f:S\rightarrow X$ be a  holomorphic curve which ramifies over $D$ with order $c>(d+1)/2,$ i.e., a  constant  $c>(d+1)/2$ such that $f^*D\geq c\cdot{\rm {Supp}}f^*D.$ If $f$ satisfies the growth condition
$$\liminf_{r\rightarrow\infty}\frac{\kappa(r)r^2}{T_{f,A}(r)}=0,$$
 then $f(S)$ 
is contained in $D.$ More precisely, if $S$ is the Poincar\'e disc $($take $o$  as the center of disc$)$,  then  $f(S)$ 
is contained in $D$ provided that  
$$\limsup_{r\rightarrow\infty}\frac{r}{T_{f,A}(r)}=0.$$

\end{thb} 
Notice, if $S=\mathbb C,$   the first  growth condition in Theorem B is automatically satisfied.
To receive a degeneracy result, some growth condition is necessary.  
 It is hard, however, to  give a perfect estimate of Green functions in a general Riemannian manifold,  hence the  first growth 
condition in Theorem B is not optimal. However, we remark that the second growth condition in Theorem B is sharp  (see comparisons with the classical results in Section 1.2  below). 

\begin{thc}[=Corollary \ref{cor1}] Assume the same conditions as in Theorem {\rm A}.
Then for any holomorphic curve  $f:S\rightarrow X$ whose image is not contained in ${\rm{Supp}}D,$
we have 
\begin{eqnarray*}
T_{f,K_X(D)}(r) &\leq_{\rm exc}&
\frac{k(d+1)}{2}N^{[1]}_f(r,D)+
O\Big(\log T_{f, A}(r)-\kappa(r)r^2+\log^+\log r\Big)
\end{eqnarray*}
 for an   integer $k$ such that  $A^{\otimes k}\geq K_X(D).$  More precisely, if $S$ is the Poincar\'e disc $($take $o$  as the center of disc$)$,  then
 \begin{eqnarray*}
T_{f,K_X(D)}(r) &\leq_{\rm exc}&
\frac{k(d+1)}{2}N^{[1]}_f(r,D)+
O\Big(\log T_{f, A}(r)+r\Big)
\end{eqnarray*}
 for an   integer $k$ such that  $A^{\otimes k}\geq K_X(D).$ 
\end{thc}

If $X$ is a smooth complex projective curve,   we can replace 
 $k(d+1)/2$  by 1 in Theorem C, see Theorem \ref{thm} in Section 4. We obtain   a direct consequence of Theorem A: 
  given any holomorphic curve  $f:S\rightarrow X$ whose image is not contained in ${\rm{Supp}}D,$ there exists   a  positive constant  $c_\psi,$ depending only on $\psi$ and $(X,D)$, such that 
\begin{eqnarray*}
T_{f,K_X(D)}(r) &\leq_{\rm exc}&
c_\psi N^{[1]}_f(r,D)+
O\Big(\log T_{f, A}(r)-\kappa(r)r^2+\log^+\log r\Big).
\end{eqnarray*}
   More precisely, if $S$ is the Poincar\'e disc,  then
\begin{eqnarray*}
T_{f,K_X(D)}(r) &\leq_{\rm exc}&
c_\psi N^{[1]}_f(r,D)+
O\Big(\log T_{f, A}(r) +r\Big).
\end{eqnarray*}

  We  apply Theorem A to  Siegel modular varieties (see \cite{Sig}), which will be introduced in Section 5.  
  Let $\mathcal A_g^{[n]}$ $(n\geq 3)$ be the moduli space of principally polarized Abelian 
varieties with level-$n$  structure. Indeed,   
 let $\overline{\mathcal A}_g^{[n]}$  be  smooth compactification of $\mathcal A_g^{[n]}$ such that  $D=\overline{\mathcal A}_g^{[n]} \setminus\mathcal A_g^{[n]}$
is a normal crossing  divisor. Then  we obtain 

\begin{thd}[=Theorem \ref{mm}] 
For any holomorphic curve  $f:S\rightarrow \overline{\mathcal A}_g^{[n]}$  whose image is  not  contained in ${\rm{Supp}}D,$ we have
\begin{eqnarray*}
T_{f,K_{\overline{\mathcal A}_g^{[n]}}(D)}(r) &\leq_{\rm exc}&
\frac{(g+1)^2}{2}N^{[1]}_f(r,D)+
O\Big(\log T_{f, A}(r)-\kappa(r)r^2+\log^+\log r\Big).
\end{eqnarray*}
More precisely, if $S$ is the Poincar\'e disc $($take $o$  as the center of disc$)$,  then
\begin{eqnarray*}
T_{f,K_{\overline{\mathcal A}_g^{[n]}}(D)}(r) &\leq_{\rm exc}&
\frac{(g+1)^2}{2}N^{[1]}_f(r,D)+
O\Big(\log T_{f, A}(r)+r\Big).
\end{eqnarray*}
\end{thd}

\subsection{Comparing Theorem A with the classical results}~

We compare Theorem A obtained with the classical results in Nevanlinna theory. Without going into the  details, we  shall refer the reader to  the recent  excellent papers such as P${\rm{\breve{a}}}$un-Sibony \cite{Sibony}, Ru \cite{ru00} and Ru-Sibony \cite{Ru-Sibony}, etc., and  refer the reader to some  good books, for example,  Noguchi-Winkelmann \cite{Noguchi}, Ru  \cite{ru} and  Vojta \cite{vojta}, etc.. 
The case  $B=\mathbb C$ is simple, we can see easily that Theorem A coincides  with the classical results  by Remark \ref{re12}.
In what follows, we  compare Theorem A  with the classical results  when  $B$ is the unit disc (see Ru-Sibony \cite{Ru-Sibony}).

Let $\mathbb D$ denote the unit disc, we need to compare the Second Main Theorem of holomorphic curve $f$ from  $\mathbb D$ into $X$ under   Poincar\'e metric and standard  Euclidean metric, respectively. To avoid confusion, it is better to  use $r, \tilde r$  to stand for  the geodesic radius under   Poincar\'e metric and   standard
Euclidean metric, respectively.  Clearly, note that $0<r<\infty$ and $0<\tilde r<1.$ Combining the arguments of Ru-Sibony \cite{Ru-Sibony} with  ones of Sun \cite{sun}, we have the classical Second Main Theorem for $\mathbb D$ (equipped with standard  Euclidean metric)  
\begin{eqnarray*}
T_{f,A}(\tilde r) &\leq_{\rm exc}&
\frac{d+1}{2}N^{[1]}_f(\tilde r,D)+
O\Big(\log T_{f, A}(\tilde r)+\log \frac{1}{1-\tilde r}\Big),
\end{eqnarray*}
where $A, d$ are given in Theorem A. Note that the main error term $O(\log \frac{1}{1-\tilde r})$ is optimal in the classical Nevanlinna theory \cite{Noguchi, ru}.  Our result in Theorem A says that 
\begin{eqnarray*}
T_{f,A}(r) &\leq_{\rm exc}&
\frac{d+1}{2}N^{[1]}_f(r,D)+
O\Big(\log T_{f, A}(r)+r\Big). 
\end{eqnarray*}

$a)$ Comparing the main error terms

 Taking  $o$ as the center  of $\mathbb D.$ Let $r(x), \tilde r(x)$ denote the Riemannian distance function of a point $x$ from $o$ under Poincar\'e metric and standard  Euclidean metric,  respectively. 
 We compare the main error terms $O(r)$ and $O(\log \frac{1}{1-\tilde r}).$
 Let the radius $r$ correspond to the radius $\tilde r.$
 By the relation (\cite{ru}, Page 273)
\begin{equation}\label{ppp1}
r(x)=\log\frac{1+\tilde r(x)}{1-\tilde r(x)},
\end{equation}
it follows that  
$$r=\log \frac{1+\tilde r}{1-\tilde r}=\log\frac{1}{1-\tilde r}+O(1)$$
due to $\tilde r<1.$ This implies that the main error terms $O(r), O(\log \frac{1}{1-\tilde r})$ (under the two  metrics respectively) are actually equivalent.

$b)$ Comparing the Nevanlinna's functions

 We only compare the characteristic functions $T_{f,A}(r)$ and $T_{f,A}(\tilde r),$  
 and the other Nevanlinna's functions (proximity functions and counting functions) can be  compared similarly. 
Refer to the definition for Nevalinna's functions  in our settings in Section 2.2, and the definition for the classical Nevalinna's functions in \cite{Noguchi, ru}. 
 In the following, we will prove that $T_{f,A}(r)$ and $ T_{f,A}(\tilde r)$ are a match (and so are the other Nevanlinna's functions),  
 if  the radius $r$ is corresponding  to  the radius $\tilde r$ via the relation 
 $r=\log [(1+\tilde r)/(1-\tilde r)].$
 
Let $\Delta, \tilde\Delta$ be the  Laplace-Beltrami operators under  Poincar\'e metric and standard Euclidean metric, respectively, 
and  let $D(r), \tilde D(\tilde r)$ be the geodesic balls with radius $r, \tilde r$ centered at $o$ under the two metrics, respectively. 
Moreover,  we denote by $g_r(o, x), \tilde g_{\tilde r}(o, x)$  the Green functions (see  (\ref{ooxx})) of $\Delta/2, \tilde\Delta/2$ for $D(r), \tilde D(\tilde r)$ with Dirichlet boundary condition and a pole $o,$ respectively. 
 Let the radius $r$ correspond to the radius $\tilde r,$ then $D(r)$ corresponds to  $\tilde D(\tilde r).$ 
Notice that  
$$\tilde g_{\tilde r}(o, x)=\frac{1}{\pi}\log\frac{\tilde r}{\tilde r(x)}$$
which corresponds to the Green function 
 $$g_r(o, x)=\frac{1}{\pi}\log\frac{(e^r-1)(e^{r(x)}+1)}{(e^r+1)(e^{r(x)}-1)}$$
due to  (\ref{ppp1}).  
A direct computation shows that 
$\Delta udV=\tilde\Delta ud\tilde V$
for any smooth function $u,$ where   $dV, d\tilde V$ are the    
volume elements of $\mathbb D$ under  Poincar\'e metric and standard Euclidean metric, respectively. 
Hence, we conclude that (see (\ref{xcv}) for  definition of characteristic function)
 \begin{eqnarray*}
   T_{f,A}(r)
   &=& \pi\int_{D(r)}g_r(o,x)f^*c_1(A,h) \\
   &=&-\frac{1}{4}\int_{D(r)}g_r(o,x)\Delta\log h\circ f(x)dV(x) \\
   &=& -\frac{1}{4}\int_{\tilde D(\tilde r)}\tilde g_{\tilde r}(o,x)\tilde \Delta\log h\circ f(x)d\tilde V(x) \\
    &=& \pi\int_{\tilde D(\tilde r)}\tilde g_{\tilde r}(o,x)f^*c_1(A,h) \\
     &=& \int_0^{\tilde r}\frac{dt}{t}\int_{\tilde D(\tilde t)}f^*c_1(A,h) \\
     &=& T_{f,A}(\tilde r), 
 \end{eqnarray*}
where $h$ is a Hermitian metric on $A$ such that the Chern form $c_1(A, h)>0.$
This  certifies  that $T_{f, A}(r), \tilde T_{f, A}(\tilde r)$ are a match. 
It means that the  two Second Main Theorems  are  equivalent.  

\begin{remark} Recently,  P${\rm{\breve{a}}}$un-Sibony \cite{Sibony} investigated   the value distribution of a  holomorphic mapping 
 $f: \mathcal Y\rightarrow \mathbb P^{1}(\mathbb C),$ where $\mathcal Y$ is a parabolic Riemann surface (i.e., an open Riemann surface with a parabolic exhaustion function $\sigma$). 
 An earlier work  concerning holomorphic curves from $\mathcal Y$ into $\mathbb P^{n}(\mathbb C)$
  is due to  Wu \cite{wu} (see also He-Ru \cite{HR} and Shabat \cite{Shabat}).  Let us  introduce the result of P${\rm{\breve{a}}}$un-Sibony briefly. 
Given the parabolic ball $B(r):=\{x\in\mathcal Y: \sigma(x)<r\}.$  The characteristic function  is defined by 
 $$T_{f}(r)=\int_1^r\frac{dt}{t}\int_{B(t)}f^*\omega_{FS},$$
where $\omega_{FS}$ is the Fubini-Study form on $\mathbb P^{1}(\mathbb C).$ The other Nevanlinna's functions can be  defined in this setting (\cite{Sibony}, Page 13).  P${\rm{\breve{a}}}$un-Sibony  proved that (\cite{Sibony}, Theorem 3.2)
$$(q-2)T_{f}(r)\leq_{{\rm exc}} \sum_{j=1}^qN^{[1]}_f(r, a_j)+\mathcal X_\sigma(r)+O\Big(\log T_{f}(r)+\log r\Big)$$
for distinct points $a_1,\cdots, a_q$ in  $\mathbb P^{1}(\mathbb C)$, where 
$$\mathcal X_\sigma(r)=\int_1^r|\chi_\sigma(t)|\frac{dt}{t}, \ \ \   \chi_\sigma(t)=\chi(B(t)).$$
It is expected to compare  our result  with  P${\rm{\breve{a}}}$un-Sibony's result, but it seems  difficult. 
In  the setting of P${\rm{\breve{a}}}$un-Sibony,  $\sigma$ may give more than one connected components of $B(r),$  and each connected component may be multi-connected  if the genus of $\mathcal Y$
 is  greater  than 0.   
In our setting,  however, 
the completeness of  metric implies that 
 the geodesic ball $D(r)$  
 is  simply-connected (with genus $0$). 
 Therefore,  one can hardly compare  the Nevanlinna's functions in the two settings. 
We also can  hardly compare  error terms $-\kappa(r)r^2$ (or $r$) and $\mathcal X_\sigma(r).$
In my opinion, it gives two different ways in describing the  value distribution of  curve $f:$
one is  to use the topology (Euler characteristic) of $\mathcal Y,$  the other  is  to use the geometry (Gauss curvautre) of $\mathcal Y.$

\end{remark}

\subsection{Work of Carne and Atsuji on Nevanlinna theory}~ 

In 1986, Carne \cite{carne} first noticed the relationship between Nevanlinna theory and Brownian motions, he formulated Nevanlinna's functions of a 
meromorphic function on $\mathbb C$ via the Brownian motion $X_t$ in $\mathbb C,$ where $X_t$ is generated by $\Delta_\mathbb C/2.$
By means of stochastic calculus, specially, It\^o formula and Coarea formula, 
Carne  re-obtained  the  Nevanlinna theory of meromorphic functions on $\mathbb C.$
 We would mention that Carne's method is also suitable  to the case $\mathbb C^n,$ though he  hadn't   pushed his work since then. 
 A remarkable development of work of Carne is due to Atsuji \cite{at, atsuji},  who investigated the value distribution theory of meromorphic functions on K\"ahler 
 manifolds,
 along  a line of Carne, but developed   techniques of Carne. 
 The major  work of Atsuji is  to 
 generalize  
  Nevanlinna theory to  the non-positively curved  complete K\"ahler manifolds,  
  and 
  one of his main  contributions  is the
   estimation  of    lower bounds of Green functions.  
   This work of estimation 
    leads him to obtain   
    the Calculus Lemma (\cite{atsuji},  Lemma 13)
     which is a useful tool in the study of  Nevanlinna theory. 
     By combining another estimate (\cite{atsuji}, Lemma 12),   
 he  established a Second Main Theorem (with an error term depending on the curvature of the manifolds)
  of meromorphic functions on such class of manifolds (\cite{atsuji}, Theorem 9). 
  
We compare  Theorem A  with the results of Atusji 
in case that $X=\mathbb P^1(\mathbb C)$ and  domain is the Riemann surface. 
For a general open Riemann surface $S$, Theorem A  agrees with Atsuji's result. 
However,  for $S=\mathbb D$ (equipped with Poincar\'e metric),  
Theorem A gives  a  main  error term  
$O(r),$ 
 but  Atsuji gave
 $O(r^2)$ according to his theorem (\cite{atsuji}, Theorem 9). 
 It  implies  that our results are better than  Atsuji's results.

\section{Nevanlinna's functions and First Main Theorem}

\subsection{Coarea formula and Dynkin formula}~

 Let  $(M,g)$ be a Riemannian manifold with  Laplace-Beltrami operator $\Delta_M$ associated to  $g.$  Fix $x\in M,$ denote by $B_x(r)$ the geodesic ball centered at $x$ with radius $r,$ and  by $S_x(r)$ the geodesic sphere centered at $x$ with radius $r.$
 Apply  Sard's theorem, $S_x(r)$ is a submanifold of $M$ for almost all $r>0.$
A Brownian motion $(X_t)_{t\geq0}$ (written as $X_t$ for short) in $M$
is a heat diffusion  process  generated by $\Delta_M/2$ with  transition density function $p(t,x,y)$ which is the minimal positive fundamental solution of the  heat equation
  $$\frac{\partial}{\partial t}u(t,x)-\frac{1}{2}\Delta_{M}u(t,x)=0.$$
We denote by $\mathbb P_x$ the law of $X_t$ started at $x\in M$
 and by $\mathbb E_x$ the corresponding expectation with respect to $\mathbb P_x.$ The reader may refer to \cite{bass, dong,  NN,itoo} to learn 
 more about Brownian motions.  

 \noindent\textbf{A. Coarea formula}

  Let $D$  be a bounded domain with   smooth boundary $\partial D$ in $M$.
Fix $x\in D,$  we use $d\pi^{\partial D}_x$ to denote the harmonic measure  on $\partial D$ with respect to $x.$
This measure is a probability measure.
 Set
$$\tau_D:=\inf\big\{t>0:X_t\not\in D\big\}$$
which is a stopping time.
Let $g_D(x,y)$ stand for  the Green function of $\Delta_M/2$ for  $D$ with Dirichlet boundary condition and a pole  $x$, namely
\begin{equation}\label{ooxx}
-\frac{1}{2}\Delta_{M}g_D(x,y)=\delta_x(y), \ y\in D; \ \ g_D(x,y)=0, \ y\in \partial D,
\end{equation}
where $\delta_x$ is the Dirac function.
 For $\phi\in \mathscr{C}_{\flat}(D)$
 (space of bounded continuous functions on $D$), the \emph{Coarea formula} \cite{bass} says  that
$$ \mathbb{E}_x\left[\int_0^{\tau_D}\phi(X_t)dt\right]=\int_{D}g_{D}(x,y)\phi(y)dV(y),
$$
where $dV(y)$ is the Riemannian volume element on $M.$
By Proposition 2.8 in \cite{bass}, we  have the following relation between harmonic measures and hitting times that
\begin{equation}\label{w121}
  \mathbb{E}_x\left[\psi(X_{\tau_{D}})\right]=\int_{\partial D}\psi(y)d\pi_x^{\partial D}(y)
\end{equation}
for any $\psi\in\mathscr{C}(\overline{D})$.
 Coarea formula or (\ref{w121})  works if  the set of singularities  of $\phi$ or $\psi$ is  pluripolar, 
 since the Brownian motion $X_t$ hits the pluripolar set in probability 0 (see \cite{at, atsuji, bass, 13, NN, itoo}).

 \noindent\textbf{B. Dynkin formula}

Let $u\in\mathscr{C}_\flat^2(M)$ (space of bounded $\mathscr{C}^2$-class functions on $M$), we have the famous \emph{It\^o formula} (see  \cite{at, atsuji, bass, 13, NN,itoo})
$$u(X_t)-u(x)=B\left(\int_0^t\|\nabla_Mu\|^2(X_s)ds\right)+\frac{1}{2}\int_0^t\Delta_Mu(X_s)dt, \ \ \mathbb P_x-a.s.$$
where $B_t$ is the standard  Brownian motion in $\mathbb R$ and $\nabla_M$ is gradient operator on $M$.
Notice that $B_t$ is a martingale, take expectation on both sides of the above formula, it  follows \emph{Dynkin formula} (see  \cite{at, atsuji,  bass, 13, NN, itoo})
$$ \mathbb E_x\left[u(X_{\tau_D})\right]-u(x)=\frac{1}{2}\mathbb E_x\left[\int_0^{\tau_D}\Delta_Mu(X_t)dt\right].
$$
  Dynkin formula still  works if  the set  of singularities of $u$ is pluripolar.

\subsection{Nevanlinna's functions and First Main Theorem}~

Let $S$ be an open Riemann surface with  K\"ahler form $\alpha.$   Fixing $o\in S$ as a reference point. 
  Denote by $D(r)$ the geodesic ball centered at $o$ with radius $r,$ and by $\partial D(r)$ the boundary of $D(r).$
  Moreover, one uses 
  $g_r(o,x)$ to stand for the Green function of $\Delta_S/2$ for $D(r)$ with Dirichlet boundary condition and a pole $o,$  and  $d\pi^r_o(x)$ to stand for the harmonic
measure on $\partial D(r)$ with respect to $o.$
Now, let $X_t$ be the Brownian motion in $S$ 
with generator $\Delta_S/2,$  started at  $o.$  
Set the stopping time $$\tau_r=\inf\big\{t>0: X_t\not\in D(r)\big\}.$$

Let $$f:S\rightarrow X$$  be a holomorphic curve into a compact complex  manifold $X.$ We  introduce  the generalized  Nevanlinna's functions over
  Riemann surface $S.$
Let $L\rightarrow X$
be an ample  holomorphic line bundle  equipped with  Hermitian metric $h$ such that the Chern form
$c_1(L,h)>0.$   Let's  define the \emph{characteristic function} of $f$ with respect to $L$  by
 \begin{eqnarray}\label{xcv}
   T_{f,L}(r)
   &=& \pi\int_{D(r)}g_r(o,x)f^*c_1(L,h) \\
   &=&-\frac{1}{4}\int_{D(r)}g_r(o,x)\Delta_S\log h\circ f(x)dV(x), \nonumber
 \end{eqnarray}
 where $dV(x)$ is the Riemannian volume measure of $S.$ It can be easily checked that $T_{f,L}(r)$ is independent
 of the choices of  metrics on $L,$ up to a bounded term.
 Since a holomorphic line bundle
 can be represented  as the difference of two  ample holomorphic line bundles,  the definition of $T_{f,L}(r)$ can  extend to
 an arbitrary holomorphic line bundle.
 By Coarea formula, we obtain 
 $$T_{f,L}(r)=-\frac{1}{4}\mathbb E_o\left[\int_0^{\tau_r}\Delta_S\log h\circ f(X_t)dt\right].$$
A divisor can   be also expressed  as the difference of two ample divisors.
  The Weil function of  an effective divisor $D$  is well defined  by
$$\lambda_D(x)=-\log\|s_D(x)\|$$
up to a bounded term, where $s_D$ is the canonical section of the line bundle $\mathscr O_X(D)$ over $X,$ 
namely,  $s_D$ is locally written as $s_D=\tilde{s}_De,$ where $e$ is a local holomorphic frame of $\mathscr O_X(D),$ and $\tilde s_D$ is a local defining function of $D$. Note that $\tilde s_D$ is a holomorphic function due to that  $D$ is effective. We define 
the \emph{proximity function} of $f$ with respect to $D$   by
 $$m_f(r,D)=\int_{\partial D(r)}\lambda_D\circ f(x)d\pi^r_o(x).$$
 By (\ref{w121}), one has
  $$m_f(r,D)=\mathbb E_o\big[\lambda_D\circ f(X_{\tau_r})\big].$$
The   \emph{counting function}
 of $f$ with respect to $D$  is defined by
\begin{eqnarray*}
N_f(r,D)
&=& \pi \sum_{f^*D\cap D(r)}g_r(o,x) \\
&=& \pi\int_{D(r)}g_r(o,x)dd^c\big{[}\log|\tilde{s}_D\circ f(x)|^2\big{]} \\
&=&\frac{1}{4}\int_{D(r)}g_r(o,x)\Delta_S\log|\tilde{s}_D\circ f(x)|^2dV(x).
\end{eqnarray*}
   We   can also define the \emph{truncated counting function} (i.e., the simple counting function without counting multiplicities) of $f$ with respect to $D$ in a similar way,  denote it by  $N^{]1[}_f(r,D).$ 

\begin{remark}\label{re12}  The above  definition of Nevanlinna's functions is  natural. In case  $S=\mathbb C,$ the Green function
for $D(r)$ is computed as $(\log\frac{r}{|z|})/\pi,$
 and the harmonic measure on $\partial D(r)$ is computed as $d\theta/2\pi.$ By integration by part,  we can see that
our Nevanlinna’s functions coincide with the classical ones.
\end{remark}
Consider a holomorphic curve $f:S\rightarrow X$
 such that $f(o)\not\in {\rm{Supp}}D,$ where $D$ is an effective  divisor on $X.$
Apply Dynkin formula to $\lambda_D\circ f(x),$  it follows that 
 $$\mathbb E_o\big[\lambda_D\circ f(X_{\tau_r})\big]-\lambda_D\circ f(o)=\frac{1}{2}
 \mathbb E_o\left[\int_0^{\tau_r}\Delta_S\lambda_D\circ f(X_t)dt\right].$$
The first term on the left hand side of the above equality is equal to $m_f(r,D),$ and the term on
the right hand side equals
$$\frac{1}{2}\mathbb E_o\left[\int_0^{\tau_r}\Delta_S\lambda_D\circ f(X_t)dt\right]=
\frac{1}{2}\int_{D(r)}g_r(o,x)\Delta_S\log\frac{1}{\|s_D\circ f(x)\|}dV(x)$$
due to Coarea formula. Since $\|s_D\|^2=h|\tilde{s}_D|^2,$ where $h$ is a Hermitian metric
on $\mathscr O_X(D)$ and $\tilde s_D$ is a local defining function of $D$ given as above,  then 
\begin{eqnarray*}
\frac{1}{2}\mathbb E_o\left[\int_0^{\tau_r}\Delta_S\lambda_D\circ f(X_t)dt\right] &=&
-\frac{1}{4}\int_{D(r)}g_r(o,x)\Delta_S\log h\circ f(x)dV(x) \\
&&-\frac{1}{4}\int_{D(r)}g_r(o,x)\Delta_S\log|\tilde{s}_D\circ f(x)|^2dV(x) \\
&=& T_{f, \mathscr O_X(D)}(r)-N_f(r,D).
\end{eqnarray*}
To conclude, we obtain 
\begin{theorem}[First Main Theorem]\label{fmt} Let $X$ be a compact complex manifold.  Let $f: S\rightarrow X$ be a holomorphic curve  such that $f(o)\not\in {\rm{Supp}}D,$ where $D$ is an effective  divisor on $X.$  Then  
  $$T_{f, \mathscr O_X(D)}(r)=m_f(r,D)+N_f(r,D)+O(1).$$
  \end{theorem}
  
\section{Logarithmic Derivative Lemma}

Let $S$ be a  complete open Riemann surface with  Gauss curvature  $K_S\leq0$ associated to   Hermitian metric $g.$ In this section, we assume that 
   $S$ is simply connected. Recall that
    $$\tau_r=\inf\big\{t>0: X_t\not\in D(r)\big\}.$$
\subsection{Calculus Lemma}~

Let $\kappa$ be given by (\ref{kappa}). As  noted before,
$\kappa$ is a non-positive and decreasing continuous function  on $[0,\infty).$
  Consider the ordinary differential equation
  \begin{equation}\label{G}
    G''(t)+\kappa (t)G(t)=0; \ \ \ G(0)=0, \ \ G'(0)=1.
  \end{equation}
 Comparing (\ref{G})  with the equation $y''(t)+\kappa(0)y(t)=0$ under the same  initial conditions,
then  $G$ can be easily estimated  as
$$G(t)=t \ \ \text{for}  \ \kappa\equiv0; \ \ \ G(t)\geq t \ \ \text{for} \ \kappa\not\equiv0.$$
This implies that
\begin{equation}\label{vvvv}
  G(r)\geq r \ \ \text{for} \ r\geq0; \ \ \ \int_1^r\frac{dt}{G(t)}\leq\log r \ \ \text{for} \ r\geq1.
\end{equation}
On the other hand,   rewrite (\ref{G}) in the form
$$\log'G(t)\cdot\log'G'(t)=-\kappa(t),$$
where 
$$\log'G(t)=\frac{G'(t)}{G(t)}, \ \ \   \log'G'(t)=\frac{G''(t)}{G'(t)}.$$
Since $G(t)\geq t$ is increasing,
then the decrease and non-positivity of $\kappa$ imply that for each fixed $t,$ $G$  must satisfy one of the following two inequalities
$$\log'G(t)\leq\sqrt{-\kappa(t)} \ \ \text{for} \ t>0; \ \ \ \log'G'(t)\leq\sqrt{-\kappa(t)} \ \ \text{for} \ t\geq0.$$
Since $G(t)\rightarrow0$ as $t\rightarrow0,$ then by integration we obtain 
\begin{equation}\label{v2}
  G(r)\leq r\exp\big(r\sqrt{-\kappa(r)}\big) \ \  \text{for} \ r\geq0.
\end{equation}

\begin{lemma}[\cite{atsuji}]\label{zz} Let $\eta>0$ be a constant. Then there is  a constant $C>0$ such that for
$r>\eta$ and $x\in D(r)\setminus \overline{D(\eta)}$
  $$g_r(o,x)\int_{\eta}^r\frac{dt}{G(t)}\geq C\int_{r(x)}^r\frac{dt}{G(t)}$$
 holds,  where  $G$ is defined by {\rm{(\ref{G})}} and $r(x)$ is the Riemannian distance function of $x$ from $o$ on $S.$ 
\end{lemma}

\begin{lemma}[Borel's lemma]\label{cal1} Let $u$ be a monotone increasing function on $[0,\infty)$ such that $u(r_0)>1$ for some $r_0\geq0.$
Then for any $\delta>0,$ there exists a set $E_\delta\subseteq[0,\infty)$ of finite Lebegue measure such that
$$u'(r)\leq u(r)\log^{1+\delta}u(r)$$
holds for $r>0$ outside $E_\delta.$
\end{lemma}
\begin{proof} Since $u$ is monotone increasing,  then
$u'(r)$ exists for almost all $r\geq0.$
Set 
$S=\left\{r\geq0: u'(r)>u(r)\log^{1+\delta}u(r)\right\}.$
We have
 \begin{eqnarray*}
\int_Sdr &\leq&\int_0^{r_0}dr+\int_{S\setminus[0,r_0]}dr 
\leq r_0+\int_{r_0}^\infty\frac{u'(r)}{u(r)\log^{1+\delta}u(r)}dr <\infty.
 \end{eqnarray*}
\end{proof}

\begin{theorem}[Calculus Lemma]\label{cal}
 Let $k\geq0$ be a locally integrable  function on $S$ such that it is locally bounded at $o\in S.$
 Then for any $\delta>0,$ there is a constant $C>0$ independent of $k,\delta,$ and a subset $E_\delta\subseteq(1,\infty)$ of finite Lebesgue measure such that
$$
\mathbb E_o\big{[}k(X_{\tau_r})\big{]}
\leq \frac{F(\hat{k},\kappa,\delta)e^{r\sqrt{-\kappa(r)}}\log r}{2\pi C}\mathbb E_o\left[\int_0^{\tau_{r}}k(X_{t})dt\right]
$$
  holds for $r>1$ outside $E_\delta,$  where  $\kappa$ is defined by $(\ref{kappa})$ and $F$ is defined by
$$
F\big{(}\hat{k},\kappa, \delta\big{)}
=\Big[\log^+\hat{k}(r)\cdot\log^+\Big(re^{r\sqrt{-\kappa(r)}}\hat{k}(r)\big(\log^{+}\hat{k}(r)\big)^{1+\delta}\Big)\Big]^{1+\delta} \ \ \ \
$$with
$$\hat k(r)=\frac{\log r}{C}\mathbb E_o\left[\int_0^{\tau_{r}}k(X_{t})dt\right].$$
Moreover, we have the estimate
 \begin{eqnarray*}
&&\log^+ F(\hat{k},\kappa,\delta) \\
&\leq& O\Big(\log^+\log^+ \mathbb E_o\left[\int_0^{\tau_{r}}k(X_{t})dt\right]+\log^+r\sqrt{-\kappa(r)}+\log^+\log r\Big).
 \end{eqnarray*}
\end{theorem}
\begin{proof} The argument refers to Atsuji \cite{atsuji}. 
From  \cite{Deb}, the simple-connectedness and non-positivity of  sectional curvature of $S$ imply
$2\pi rd\pi^r_{o}(x)\leq d\sigma_r(x),$
where  $d\sigma_r(x)$ is the  volume element on $\partial D(r)$ which is induced by the volume  element $dV(x)$ on $S.$
By  Lemma \ref{zz} and (\ref{vvvv}), we have
\begin{eqnarray*}
   \mathbb E_o\left[\int_0^{\tau_{r}}k(X_{t})dt\right] &=&
   \int_{D(r)}g_r(o,x)k(x)dV(x) \\ &=&\int_0^rdt\int_{\partial D(t)}g_r(o,x)k(x)d\sigma_t(x)  \\
&\geq& C\int_0^r\frac{\int_t^rG^{-1}(s)ds}{\int_1^rG^{-1}(s)ds}dt\int_{\partial D(t)}k(x)d\sigma_t(x) \\
&\geq& \frac{C}{\log r}\int_0^rdt\int_t^r\frac{ds}{G(s)}\int_{\partial D(t)}k(x)d\sigma_t(x), \\
\mathbb E_o\left[k(X_{\tau_r})\right]&=&\int_{\partial D(r)}k(x)d\pi_o^r(x)\leq\frac{1}{2\pi r}\int_{\partial D(r)}k(x)d\sigma_r(x).
\end{eqnarray*}
 Thus,
  \begin{eqnarray}\label{fr}
\mathbb E_o\left[\int_0^{\tau_{r}}k(X_{t})dt\right]&\geq&\frac{C}{\log r}\int_0^rdt\int_t^r\frac{ds}{G(s)}\int_{\partial D(t)}k(x)d\sigma_t(x), \nonumber \\
  \mathbb E_o[k(X_{\tau_r})]&\leq&\frac{1}{2\pi r}\int_{\partial D(r)}k(x)d\sigma_r(x).
\end{eqnarray}
 Set
$$\Lambda(r)=\int_0^rdt\int_t^r\frac{ds}{G(s)}\int_{\partial D(t)}k(x)d\sigma_t(x).$$
Then 
\begin{equation*}
 \Lambda(r)\leq\frac{\log r}{C}\mathbb E_o\left[\int_0^{\tau_{r}}k(X_{t})dt\right]=\hat{k}(r). \ \  \
\end{equation*}
Since
$$\Lambda'(r)=\frac{1}{G(r)}\int_0^rdt\int_{\partial D(t)}k(x)d\sigma_t(x), \ \ \ \ \ $$
 it yields from (\ref{fr}) that
\begin{equation*}
  \mathbb E_o\big{[}k(X_{\tau_r})\big{]}\leq\frac{1}{2\pi r}\frac{d}{dr}\left(\Lambda'(r)G(r)\right).
\end{equation*}
By Lemma \ref{cal1} twice and (\ref{v2}), we see that  for any $\delta>0$
\begin{eqnarray*}
   && \frac{d}{dr}\left(\Lambda'(r)G(r)\right) \\
&\leq& G(r)\Big[\log^+\Lambda(r)\cdot\log^+\left(G(r)\Lambda(r)\big\{\log^+\Lambda(r)\big\}^{1+\delta}\right)\Big]^{1+\delta}\Lambda(r)  \\
&\leq& re^{r\sqrt{-\kappa(r)}}\Big[\log^+\hat k(r)\cdot\log^+\Big(re^{r\sqrt{-\kappa(r)}}\hat k(r)\big(\log^+\hat k(r)\big)^{1+\delta}\Big)\Big]^{1+\delta} \hat k(r) \\
&=& \frac{F\big{(}\hat k,\kappa,\delta\big{)}re^{r\sqrt{-\kappa(r)}}\log r}{C}\mathbb E_o\left[\int_0^{\tau_{r}}k(X_{t})dt\right] \ \ \
\end{eqnarray*}
holds outside a subset $E_\delta\subseteq(1,\infty)$ of finite Lebesgue measure.
Thus,
\begin{eqnarray*}
\mathbb E_o\big{[}k(X_{\tau_r})\big{]}
&\leq&\frac{F\big{(}\hat k,\kappa,\delta\big{)}e^{r\sqrt{-\kappa(r)}}\log r}{2\pi C}\mathbb E_o\left[\int_0^{\tau_{r}}k(X_{t})dt\right].
\end{eqnarray*}
Therefore, we get the desired  inequality.
Indeed, for $r>1$  we compute that
$$
\log^+ F(\hat k,\kappa,\delta)
\leq O\Big(\log^+\log^+\hat k(r)+\log^+r\sqrt{-\kappa(r)}+\log^+\log r\Big)$$
and
\begin{eqnarray*}
  \log^+\hat k(r)&\leq& \log^+ \mathbb E_o\left[\int_0^{\tau_{r}}k(X_{t})dt\right]+\log^+\log r.
  \end{eqnarray*}
 We have the desired estimate. The proof is competed.
 \end{proof}

\subsection{Logarithmic Derivative Lemma}~

Note from \cite{cc} that each open Riemann surface admits a nowhere-vanishing holomorphic  global vector field.  Fix    a nowhere-vanishing global holomorphic  vector field  $\mathfrak X$ over $S.$ 
Let $\psi$ be a meromorphic function on $(S,g).$
The norm of the gradient of $\psi$ is defined by
\begin{equation}\label{ddtt}
\|\nabla_S\psi\|^2=\frac{2}{g}\left|\frac{\partial\psi}{\partial z}\right|^2
\end{equation}
in a local holomorphic coordinate $z.$ Locally, we may write $\psi=\psi_1/\psi_0,$ where $\psi_0,\psi_1$ are local holomorphic functions without common zeros. Regard $\psi$  as a holomorphic mapping into $\mathbb P^1(\mathbb C)$  by
$x\mapsto[\psi_0(x):\psi_1(x)].$
Define
$$
T_\psi(r)=\frac{1}{4}\int_{D(r)}g_r(o,x)\Delta_S\log\big{(}|\psi_0(x)|^2+|\psi_1(x)|^2\big{)}dV(x),$$
i.e., $T_\psi(r):=T_{f, \mathscr O_{\mathbb P^1(1)}}(r).$ Also,  define 
$T(r,\psi):=m(r,\psi)+N(r,\psi),
$
where  
\begin{eqnarray*}
m(r,\psi)&=&\int_{\partial D(r)}\log^+|\psi(x)|d\pi^r_o(x), \\
N(r,\psi)&=&\pi \sum_{x\in (\psi)_\infty\cap D(r)}g_r(o,x).
\end{eqnarray*}
Let
  $i:\mathbb C\hookrightarrow\mathbb P^1(\mathbb C)$ be an inclusion defined by
 $z\mapsto[1:z].$  Via the pull-back by $i,$ we have a (1,1)-form $i^*\omega_{FS}=dd^c\log(1+|\zeta|^2)$ on $\mathbb C,$
 where $\zeta:=w_1/w_0$ and $[w_0:w_1]$ is
the homogeneous coordinate system of $\mathbb P^1(\mathbb C).$ The characteristic function of $\psi$ with respect to $i^*\omega_{FS}$ is defined by
$$\hat{T}_\psi(r) = \frac{1}{4}\int_{D(r)}g_r(o,x)\Delta_S\log(1+|\psi(x)|^2)dV(x).$$
Clearly, $\hat{T}_\psi(r)\leq T_\psi(r).$
We adopt the spherical distance $\|\cdot,\cdot\|$ on  $\mathbb P^1(\mathbb C),$ the proximity function of $\psi$  with respect to
$a\in \mathbb P^1(\mathbb C)$
is defined by
$$\hat{m}_\psi(r,a)=\int_{\partial D(r)}\log\frac{1}{\|\psi(x),a\|}d\pi_o^r(x).$$
 It is easy to check that  $m(r,\psi)=\hat{m}_\psi(r,\infty)+O(1)$ which follows that 
 $$
   T(r,\psi)=\hat{T}_\psi(r)+O(1).$$ 
 Similarly as Theorem \ref{fmt},  by Dykin formula and Coarea formula,  it is trivial to show the First Main Theorem   
   $$T\Big(r,\frac{1}{\psi-a}\Big)= T(r,\psi)+O(1).
$$
 Therefore, we arrive at
\begin{equation}\label{relation}
  T(r,\psi)+O(1)=\hat{T}_\psi(r)\leq T_\psi(r)+O(1).
\end{equation}
\begin{theorem}[LDL]\label{ldl2} Let $\psi$ be a nonconstant meromorphic function on $S.$ Then
\begin{eqnarray*}
m\Big(r,\frac{\mathfrak{X}^k(\psi)}{\psi}\Big)  &\leq_{\rm exc}& \frac{5k}{4}\log T(r,\psi)+O\Big(\log^+\log T(r,\psi)-\kappa(r)r^2
+\log^+\log r\Big),
\end{eqnarray*}
where $\mathfrak X^j=\mathfrak X\circ\mathfrak X^{j-1}$ with $\mathfrak X^0=Id,$   and
$\kappa$ is defined by $(\ref{kappa}).$ More precisely, if $S$ is the Poincar\'e disc $($take $o$ as the  center of disc$)$,  then 
\begin{eqnarray*}
m\Big(r,\frac{\mathfrak{X}^k(\psi)}{\psi}\Big)  &\leq_{\rm exc}& \frac{5k}{4}\log T(r,\psi)+O\Big(\log^+\log T(r,\psi)+r\Big).
\end{eqnarray*}
\end{theorem}
 Take a singular form 
$$\Phi=\frac{1}{|\zeta|^2(1+\log^2|\zeta|)}\frac{\sqrt{-1}}{4\pi^2}d\zeta\wedge d\bar \zeta$$
on $\mathbb P^1(\mathbb C).$
A direct computation gives that
\begin{equation*}\label{}
\int_{\mathbb P^1(\mathbb C)}\Phi=1, \ \ \ 2\pi\psi^*\Phi=\frac{\|\nabla_S\psi\|^2}{|\psi|^2(1+\log^2|\psi|)}\alpha,
\end{equation*}
where $\alpha=(g\sqrt{-1}/\pi) dz\wedge d\bar{z}$
 is the K\"ahler form of $S.$ Set
\begin{equation*}\label{ffww}
  T_\psi(r,\Phi)=\frac{1}{2\pi}\int_{D(r)}g_r(o,x)\frac{\|\nabla_S\psi\|^2}{|\psi|^2(1+\log^2|\psi|)}(x)dV(x).
\end{equation*}
By Fubini's theorem
\begin{eqnarray*}
T_\psi(r,\Phi)
&=&\int_{D(r)}g_r(o,x)\frac{\psi^*\Phi}{\alpha}dV(x)  \\
&=&\pi\int_{\zeta\in\mathbb P^1(\mathbb C)}\Phi\sum_{(\psi-\zeta)_0\cap D(r)}g_r(o,x) \\
&=&\int_{\zeta\in\mathbb P^1(\mathbb C)}N_{\psi}(r,\zeta)\Phi
\leq T(r,\psi)+O(1),
\end{eqnarray*}
which follows that 
\begin{equation}\label{get}
  T_\psi(r,\Phi)\leq T(r,\psi)+O(1).
\end{equation}
\begin{lemma}\label{999a} Assume  that $\psi(x)\not\equiv0.$ Then 
\begin{eqnarray*}
  && \frac{1}{2}\mathbb E_o\left[\log^+\frac{\|\nabla_S\psi\|^2}{|\psi|^2(1+\log^2|\psi|)}(X_{\tau_r})\right] \\
  &\leq_{\rm exc}&\frac{1}{2}\log T(r,\psi)+O\Big{(}\log^+\log T(r,\psi)+r\sqrt{-\kappa(r)}+\log^+\log r\Big{)},
\end{eqnarray*}
 where $\kappa$ is defined by $(\ref{kappa}).$
\end{lemma}
\begin{proof} By Jensen's inequality, it is clear that
\begin{eqnarray*}
   \mathbb E_o\left[\log^+\frac{\|\nabla_S\psi\|^2}{|\psi|^2(1+\log^2|\psi|)}(X_{\tau_r})\right]
   &\leq&  \mathbb E_o\left[\log\Big{(}1+\frac{\|\nabla_S\psi\|^2}{|\psi|^2(1+\log^2|\psi|)}(X_{\tau_r})\Big{)}\right] \nonumber \\
    &\leq& \log^+\mathbb E_o\left[\frac{\|\nabla_S\psi\|^2}{|\psi|^2(1+\log^2|\psi|)}(X_{\tau_r})\right]+O(1). \nonumber
\end{eqnarray*}
By Theorem \ref{cal}  with Coarea formula and (\ref{get})
\begin{eqnarray*}
   && \log^+\mathbb E_o\left[\frac{\|\nabla_S\psi\|^2}{|\psi|^2(1+\log^2|\psi|)}(X_{\tau_r})\right]  \\
   &\leq_{\rm exc}& \log^+\mathbb E_o\left[\int_0^{\tau_r}\frac{\|\nabla_S\psi\|^2}{|\psi|^2(1+\log^2|\psi|)}(X_{t})dt\right]
   +\log \frac{F(\hat{k},\kappa,\delta)e^{r\sqrt{-\kappa(r)}}\log r}{2\pi C}
    \\
   &\leq& \log T_\psi(r,\Phi)+\log F(\hat k,\kappa,\delta)+r\sqrt{-\kappa(r)}+\log^+\log r+O(1) \\
    &\leq& \log T(r,\psi)+O\Big(\log^+\log^+\hat k(r)+r\sqrt{-\kappa(r)}+\log^+\log r\Big),
\end{eqnarray*}
where
$$\hat k(r)=\frac{\log r}{C}\mathbb E_o\left[\int_0^{\tau_{r}}\frac{\|\nabla_S\psi\|^2}{|\psi|^2(1+\log^2|\psi|)}(X_{t})dt\right].$$
Indeed, we note that
$$
\hat k(r)=\frac{2\pi \log r}{C}T_\psi(r,\Phi)\leq \frac{2\pi\log r}{C} T(r,\psi). \ \ \ \ \  \ 
$$
Hence, we  have the desired inequality.
\end{proof}
\begin{lemma}\label{ldl1} Let $\psi$ be a nonconstant meromorphic function on $S.$  Then
\begin{eqnarray*}
m\Big(r,\frac{\mathfrak{X}(\psi)}{\psi}\Big)&\leq_{\rm exc}&\frac{5}{4}\log T(r,\psi)+O\Big{(}\log^+\log T(r,\psi)-\kappa(r)r^2+\log^+\log r\Big{)},
\end{eqnarray*}
where $\kappa$ is defined by $(\ref{kappa}).$ More precisely, if $S$ is the Poincar\'e disc $($take $o$  as the center of disc$)$,  then  \begin{eqnarray*}
m\Big(r,\frac{\mathfrak{X}(\psi)}{\psi}\Big)&\leq_{\rm exc}&\frac{5}{4}\log T(r,\psi)+O\Big{(}\log^+\log T(r,\psi)+r\Big{)}.
\end{eqnarray*}
\end{lemma}
\begin{proof} In terms of a local holomorphic coordinate $z,$  one can write $\mathfrak{X}=a\frac{\partial}{\partial z}$ satisfying   $\|\mathfrak{X}\|^2=g|a|^2,$ where $a$ is a local holomorphic function and $g$ is the Hermitian metric on $S.$  Then 
\begin{eqnarray*}
 m\Big(r,\frac{\mathfrak{X}(\psi)}{\psi}\Big)
&=& \int_{\partial D(r)}\log^+\frac{|\mathfrak{X}(\psi)|}{|\psi|}(x)d\pi^r_o(x) \\
&\leq& \frac{1}{2}\int_{\partial D(r)}\log^+\frac{|\mathfrak{X}(\psi)|^2}{\|\mathfrak{X}\|^2|\psi|^2(1+\log^2|\psi|)}(x)d\pi^r_o(x) \\
&&+
\frac{1}{2}\int_{\partial D(r)}\log(1+\log^2|\psi(x)|)d\pi^r_o(x) \\
&& +\frac{1}{2}\int_{\partial D(r)}\log^+\|\mathfrak{X}_x\|^2d\pi^r_o(x) \\
&:=& A+B+C.
\end{eqnarray*}
We handle $A,
B,C$ respectively.  For $A,$ it follows from  (\ref{ddtt}) and Lemma \ref{999a} that
\begin{eqnarray*}
A&=& \frac{1}{2}\int_{\partial D(r)}\log^+\frac{|a|^2\left|\frac{\partial \psi}{\partial z}\right|^2}{g|a|^2|\psi|^2(1+\log^2|\psi|)}(x)d\pi^r_o(x) \\
&=& \frac{1}{4}\int_{\partial D(r)}\log^+\frac{\|\nabla_S\psi\|^2}{|\psi|^2(1+\log^2|\psi|)}(x)d\pi^r_o(x) \\
 &\leq_{\rm exc}&\frac{1}{4}\log T(r,\psi)+O\Big{(}\log^+\log T(r,\psi)+r\sqrt{-\kappa(r)}+\log^+\log r\Big{)}.
\end{eqnarray*}
For $B,$ the Jensen's inequality and First Main Theorem implies that
\begin{eqnarray*}
B &\leq&  \int_{\partial D(r)}\log\Big(1+\log^+|\psi(x)|+\log^+\frac{1}{|\psi(x)|}\Big)d\pi^r_o(x) \\
&\leq& \log\int_{\partial D(r)}\Big(1+\log^+|\psi(x)|+\log^+\frac{1}{|\psi(x)|}\Big)d\pi^r_o(x) \\
&=&   \log\Big(m(r, \psi)+m\big(r, \frac{1}{\psi}\big)+1\Big) \\
&\leq& \log T(r,\psi)+O(1).
\end{eqnarray*}
Finally, we estimate $C.$ Since $\mathfrak X$ never vanishes, we obtain  $\|\mathfrak{X}\|>0.$ Since $S$ is non-positively curved and $a$ is holomorphic, then $\log\|\mathfrak{X}\|$ is subharmonic, i.e., $\Delta_S\log\|\mathfrak X\|\geq0.$
It is easy to check that 
\begin{equation}\label{h1}
\Delta_S\log^+\|\mathfrak{X}\|\leq \Delta_S\log \|\mathfrak{X}\|
\end{equation}
for $x\in S$ satisfying $\|\mathfrak{X}_x\|\neq1,$ and  
\begin{equation}\label{h2}
\log^+\|\mathfrak{X}_x\|=0
\end{equation}
for $x\in S$ satisfying $\|\mathfrak{X}_x\|\leq1.$ Note that Dynkin formula cannot be  directly applied to $\log^+\|\mathfrak{X}_x\|,$ but by virtue of (\ref{h1}) and (\ref{h2}), it is not hard to 
verify
\begin{eqnarray*}\label{ok}
\ \ \ \ \ C &=& \frac{1}{2}\mathbb E_o\left[\log^+\|\mathfrak{X}(X_{\tau_r})\|^2\right] \\
&\leq&  \frac{1}{4}\mathbb E_o\left[\int_0^{\tau_r}\Delta_S\log\|\mathfrak{X}(X_t)\|^2dt \right]+O(1) \nonumber \\
 &=& \frac{1}{4}\mathbb E_o\left[\int_0^{\tau_r}\Delta_S\log g(X_t)dt \right]+
 \frac{1}{4}\mathbb E_o\left[\int_0^{\tau_r}\Delta_S\log|a(X_t)|^2dt \right]+O(1) \nonumber \\
 &=& -\frac{1}{2}\mathbb E_o\left[\int_0^{\tau_r}K_S(X_t)dt \right]+O(1) \nonumber \\
 &\leq& -2\kappa(r)\mathbb E_o\big[\tau_r\big]+O(1), \nonumber
\end{eqnarray*}
where we utilize  the fact $K_S=-(\Delta_S\log g)/2.$ Therefore, we deduce the first conclusion of the lemma by using  $\mathbb E_o\left[\tau_r\right]\leq 2r^2,$ due
to Lemma \ref{yyyy} below. 

If $S$ is the Poincar\'e disc,  then we have $g=2/(1-|z|^2)^{2}$ and $\kappa(r)\equiv-1.$ Take $\mathfrak X=\partial/\partial z,$ then
$C$  can be   estimated  as follows
\begin{eqnarray*}
C &=& \frac{1}{2}\mathbb E_o\left[\log^+\|\mathfrak{X}(X_{\tau_r})\|^2\right] \\
&=&  \frac{1}{2}\int_{\partial D(r)}\log\frac{2}{(1-(e^r-1)^2/(e^r+1)^2)^2}\frac{d\theta}{2\pi} \\
&=& \log\frac{\sqrt2}{1-(e^r-1)^2/(e^r+1)^2} \\
&\leq& r+O(1).
\end{eqnarray*}
This proves the second conclusion of the lemma.
\end{proof}
 \begin{lemma}\label{yyyy} Let $\tau_r$ be defined as above.  Then
 $$\mathbb E_o\big[\tau_r\big]\leq 2r^2.$$
\end{lemma}
\begin{proof}  The argument follows essentially from  Atsuji \cite{atsuji}, but here we provide a simpler proof though a rougher estimate.
Let $X_t$ be the Brownian motion in $S$ started at  $o\not=o_1,$ here $o_1\in D(r).$  Let  $r_1(x)$ be the Riemannian distance function of $x$ from $o_1.$
Apply It$\rm{\hat{o}}$ formula to $r_1(x)$
\begin{equation}\label{kiss}
  r_1(X_t)-r_1(X_0)=B_t-L_t+\frac{1}{2}\int_0^t\Delta_Sr_1(X_s)ds,
\end{equation}
here $B_t$ is the standard Brownian motion in $\mathbb R,$ and $L_t$ is a local time on cut locus of $o,$ an increasing process
which increases only at cut loci of $o.$ Since $S$ is simply connected and  non-positively  curved, then
$$\Delta_Sr_1(x)\geq\frac{1}{r_1(x)}, \ \ L_t\equiv0.$$
By (\ref{kiss}), we arrive at
$$r_1(X_t)\geq B_t+\frac{1}{2}\int_0^t\frac{ds}{r_1(X_s)}.$$
Let $t=\tau_r$ and take expectation on both sides of the above inequality, then it yields that 
$$\max_{x\in \partial D(r)} r_1(x)\geq \frac{\mathbb E_o[\tau_r]}{2\max_{x\in \partial D(r)} r_1(x)}.$$
Let $o'\rightarrow o,$ 
we are led to the desired inequality.
\end{proof}

\emph{Proof of Theorem $\ref{ldl2}$}

The assertion  can be confirmed by  
$$m\Big(r, \frac{\mathfrak{X}^k(\psi)}{\psi}\Big)\leq \sum_{j=1}^k  m\Big(r,\frac{\mathfrak{X}^{j}(\psi)}{\mathfrak{X}^{j-1}(\psi)}\Big)
$$ together with the following Lemma \ref{hello1}.

\begin{lemma}\label{hello1} We have
\begin{eqnarray*}
m\Big(r,\frac{\mathfrak{X}^{k+1}(\psi)}{\mathfrak{X}^{k}(\psi)}\Big)
&\leq_{\rm exc}& \frac{5}{4}\log T(r,\psi)+O\Big(\log^+\log T(r,\psi)-\kappa(r)r^2+\log^+\log r\Big),
\end{eqnarray*}
 where $\kappa$ is defined by $(\ref{kappa}).$ More precisely, if $S$ is the Poincar\'e disc $($take $o$  as the center of disc$)$,  then we have  \begin{eqnarray*}
m\Big(r,\frac{\mathfrak{X}^{k+1}(\psi)}{\mathfrak{X}^{k}(\psi)}\Big)
&\leq_{{\rm{exc}}}& \frac{5}{4}\log T(r,\psi)+O\Big(\log^+\log T(r,\psi)+r\Big).
\end{eqnarray*}

\end{lemma}
\begin{proof}
We claim that 
\begin{eqnarray}\label{hello}
\ \ T\big{(}r,\mathfrak{X}^k(\psi)\big{)}
&\leq& 2^kT(r,\psi)+O\Big(\log T(r,\psi)-\kappa(r)r^2+\log^+\log r\Big).
\end{eqnarray}
By virtue of Lemma \ref{ldl1}, when $k=1$
\begin{eqnarray*}
T(r,\mathfrak{X}(\psi))
&=& m(r,\mathfrak{X}(\psi))+N(r,\mathfrak{X}(\psi)) \\
&\leq& m(r,\psi)+2N(r,\psi)+m\Big(r,\frac{\mathfrak{X}(\psi)}{\psi}\Big) \\
&\leq& 2T(r,\psi)+m\Big(r,\frac{\mathfrak{X}(\psi)}{\psi}\Big) \\
&\leq& 2T(r,\psi)+O\Big(\log T(r,\psi)-\kappa(r)r^2+\log^+\log r\Big)
\end{eqnarray*}
holds for $r>1$ outside a set of finite Lebesgue measure.
Assuming now that the claim holds for $k\leq n-1.$ By induction, we only need to prove the claim in the case when $k=n.$ From the claim for $k=1$ showed above and Lemma \ref{ldl1} repeatedly, we conclude that 
\begin{eqnarray*}
 T\big{(}r,\mathfrak{X}^n(\psi)\big{)}
&\leq& 2T\big{(}r,\mathfrak{X}^{n-1}(\psi)\big{)}+O\Big(\log T\big{(}r,\mathfrak{X}^{n-1}(\psi)\big{)}-\kappa(r)r^2+\log^+\log r\Big) \\
&\leq& 2^{n}T(r,\psi)+O\Big(\log T(r,\psi)-\kappa(r)r^2+\log^+\log r\Big) \\
&&
+O\Big(\log T\big{(}r,\mathfrak{X}^{n-1}(\psi)\big{)}-\kappa(r)r^2+\log^+\log r\Big) \\
&\leq& 2^nT(r,\psi)+O\Big(\log T(r,\psi)-\kappa(r)r^2+\log^+\log r\Big) \\
&&+O\left(\log T\big{(}r,\mathfrak{X}^{n-1}(\psi)\big{)}\right) \\
&& \cdots\cdots\cdots \\
&\leq& 2^nT(r,\psi)+O\Big(\log T(r,\psi)-\kappa(r)r^2+\log^+\log r\Big). \ \ \ \ \
\end{eqnarray*}
Thus,  claim (\ref{hello}) is confirmed.
 Using Lemma \ref{ldl1} and (\ref{hello}) to get
\begin{eqnarray*}
&& m\left(r,\frac{\mathfrak{X}^{k+1}(\psi)}{\mathfrak{X}^{k}(\psi)}\right) \\
&\leq& \frac{5}{4}\log T\big{(}r,\mathfrak{X}^k(\psi)\big{)}+O\left(\log^+\log T\big{(}r,\mathfrak{X}^{k}(\psi)\big{)}-\kappa(r)r^2+\log^+\log r\right)\\
&\leq& \frac{5}{4}\log T(r,\psi)+O\Big(\log^+\log T(r,\psi)-\kappa(r)r^2+\log^+\log r\Big). \ \ \ \ \
\end{eqnarray*}
This proves the first conclusion of the lemma. The second conclusion of the lemma can be proved similarly  by replacing $\kappa(r)r^2$ by $-r,$ due to the second conclusion of Lemma \ref{ldl1}.
\end{proof}

\section{Tautological inequality}

In this section, we  provide a version of tautological inequality for a general open Riemann surface $S$ from a geometric point of view by using stochastic calculus. This inequality plays an essential role in proving the main theorem. 

Let $(X,D)$ be a  smooth logarithmic pair over $\mathbb C$. Denote by $\Omega^1_X(\log D)$  the logarithmic cotangent sheaf over $X$ which is the sheaf of germs of logarithmic 1-forms with poles at most on $D,$ namely
$$\Omega^1_X(\log D)=\sum_{j=1}^s\mathscr O_X\frac{d\sigma_j}{\sigma_j}+\Omega_X^1,$$
where $\sigma_1,\cdots,\sigma_s$ are irreducible and $\sigma_1\cdots\sigma_s=0$ is a local defining equation of $D.$
Note that $\Omega^1_X(\log D)$ is  locally free.
For a sheaf $\mathscr E,$ we have the familiar  symbols (see \cite{vojta})
$$\mathbb P(\mathscr E)={\rm{\textbf{Proj}}}\bigoplus_{d\geq0}S^d\mathscr E, \ \ \ 
\mathbb V(\mathscr E)={\rm{\textbf{Spec}}}\bigoplus_{d\geq0}S^d\mathscr E,
\ \ \  {\rm{\textbf{Sym}}}^\bullet=\bigoplus_{d\geq0}S^d.$$
\ \ \ \  Associate a nonconstant holomorphic curve
$f:S\rightarrow X$
 whose image $f(S)$ is not contained in  ${\rm{Supp}}D.$ Then the curve $f$  induces  a  lifted curve
$$f':S\rightarrow \mathbb P\big(\Omega^1_X(\log D)\big)$$
which is holomorphic on $S.$
Let 
$$p: B\rightarrow\mathbb P\big(\Omega^1_X(\log D)\oplus \mathscr O_X\big)$$
be the blow-up of $\mathbb P(\Omega^1_X(\log D)\oplus \mathscr O_X)$ along the zero section 
of 
$\mathbb V(\Omega^1_X(\log D)),$ namely, the section corresponding to the projection 
$\Omega^1_X(\log D)\oplus\mathscr O_X\rightarrow\mathscr O_X.$  That is to say, 
$B$ is the closure of the graph of the induced rational mapping 
$\mathbb P\big(\Omega^1_X(\log D)\oplus\mathscr O_X\big)\dashrightarrow\mathbb P(\Omega_X^1(\log D)).$
Let $[0]$ be the exceptional 
divisor on $B.$  
There is  the  natural  lifted curve  
$$\partial f:S\rightarrow \mathbb P\big(\Omega^1_X(\log D)\oplus\mathscr O_X\big)$$
of the  curve $f.$ 
\begin{defi} Let $\phi:S\rightarrow B$ be the lift of $\partial f.$
The $D$-modified ramification counting function of a holomorphic curve 
$f:S\rightarrow X$ is  the counting function for $\phi^*[0]:$
$$N_{f,{\rm{Ram}}(D)}(r):=N_\phi(r,[0]).$$
\end{defi}

\begin{theorem} [Tautological inequality]\label{tau} Let $\mathscr A$ be an ample line sheaf over a smooth logarithmic pair $(X,D).$ Let $\mathscr O(1)$ be the tautological line sheaf over $\mathbb P(\Omega^1_X(\log D)).$ Then
\begin{eqnarray*}
&&T_{f',\mathscr O(1)}(r)+N_{f,{\rm{Ram}} (D)}(r) \\ &\leq_{\rm exc}&
N^{[1]}_f(r,D)+
O\Big(\log^+T_{f,\mathscr A}(r)-\kappa(r)r^2+\log^+\log r\Big),
\end{eqnarray*}
where $\kappa$ is defined by $(\ref{kappa}).$ More precisely, if $S$ is the Poincar\'e disc $($take $o$  as the center of disc$)$,  then 
\begin{eqnarray*}
T_{f',\mathscr O(1)}(r)+N_{f,{\rm{Ram}} (D)}(r)  &\leq_{\rm exc}&
N^{[1]}_f(r,D)+
O\Big(\log^+T_{f,\mathscr A}(r)+r\Big).
\end{eqnarray*}
\end{theorem}
We have a natural embedding $\mathbb V(\Omega^1_X(\log D))\hookrightarrow\mathbb P(\Omega^1_X(\log D)\oplus \mathscr O_X)$
that realizes $\mathbb P(\Omega^1_X(\log D)\oplus \mathscr O_X)$ as the projective closure
on fibers of $\mathbb V(\Omega^1_X(\log D)).$
Let  $[\infty]$  the  (reduced) divisor  $\mathbb P(\Omega^1_X(\log D)\oplus \mathscr O_X)\setminus \mathbb V(\Omega^1_X(\log D)).$

 In order to prove Theorem \ref{tau}, we need the following lemma
\begin{lemma}\label{key} We have
$$m_{\partial f}(r,[\infty])\leq_{\rm exc} O\Big(\log T_{f,\mathscr A}(r)-\kappa(r)r^2+\log^+\log r\Big),$$
where $\kappa$ is defined by $(\ref{kappa}).$  More precisely, if $S$ is the Poincar\'e disc $($take $o$  as the center of disc$)$,  then
$$m_{\partial f}(r,[\infty])\leq_{\rm exc} O\Big(\log T_{f,\mathscr A}(r)+r\Big).$$
\end{lemma}
\begin{proof} Without loss of generality, we may assume that $S$ is simply connected (see Remark \ref{44}).
Notice  that there exists a finite set $\mathscr H$ of rational functions on $X$ with  properties: for each point $y\in X,$
there exists a subset $\mathscr H_y\subseteq \mathscr H$ such that

 (i) $\mathscr H_y$ generates $\Omega^1_X(\log D);$

(ii) $d h/h$ is a regular section of  $\Omega^1_X(\log D)$ at $y$ for all $h\in\mathscr H_y.$

\noindent By the definition of $[\infty]$ and the compactness of $X,$ it follows that
$$\lambda_{[\infty]}\circ\partial f\leq\log^+
\max_{h\in\mathscr H}\left|\frac{\frak X(h\circ f)}{h\circ f}\right|+O(1).$$
Therefore, by the first conclusion of Theorem \ref{ldl2} we arrive at
\begin{eqnarray*}
m_{\partial f}(r,[\infty]) &\leq& C\cdot \max_{h\in\mathscr H}m\Big(r,\frac{\frak X(h\circ f)}{h\circ f}\Big)
+O(1)
\\
&\leq& O\Big(\max_{h\in\mathscr H}\log T(r,h\circ f)-\kappa(r)r^2+\log^+\log r\Big) \\
&\leq& O\Big(\log T_{f,\mathscr A}(r)-\kappa(r)r^2+\log^+\log r\Big).
\end{eqnarray*}
The second conclusion of the lemma can  be also proved by  using the second conclusion of  Theorem \ref{ldl2}.
\end{proof}
\begin{remark}\label{44}   Let $\pi:\tilde{S}\rightarrow S$ be  the (analytic) universal covering.  
Equip  $\tilde{S}$ with the  metric induced
 from the 
metric of $S.$ Then, $\tilde{S}$ is a simply-connected and complete open Riemann surface of non-positive Gauss curvature.
Take  a diffusion process $\tilde{X}_t$ in $\tilde{S}$  satisfying  that $X_t=\pi(\tilde{X}_t),$ then
$\tilde{X}_t$ is a Brownian motion with generator  $\Delta_{\tilde{S}}/2$ induced from the pull-back metric.
Now let $\tilde{X}_t$ start at $\tilde{o}\in\tilde{S}$  with $o=\pi(\tilde{o}),$  we have  
$$\mathbb E_o[\phi(X_t)]=\mathbb E_{\tilde{o}}\big{[}\phi\circ\pi(\tilde{X}_t)\big{]}$$
for $\phi\in \mathscr{C}_{\flat}(S).$ Set $$\tilde{\tau}_r=\inf\big{\{}t>0: \tilde{X}_t\not\in \tilde{D}(r)\big{\}},$$ where
$\tilde{D}(r)$ is a geodesic ball centered at $\tilde{o}$ with radius $r$ in $\tilde{S}.$
 If necessary, one can extend the filtration in probability space where $(X_t,\mathbb P_o)$ are defined so that $\tilde{\tau}_r$ is a stopping time with
 respect to a filtration where the stochastic calculus of $X_t$ works.
With the above arguments, we can suppose $S$ is simply connected, without loss of generality,  by lifting $f$ to the universal covering.
\end{remark}

\emph{Proof of  Theorem $\ref{tau}$}

The proof  essentially follows  McQuillan \cite{MQ}. Recall the blow-up
$$p: B\rightarrow\mathbb P\big(\Omega^1_X(\log D)\oplus \mathscr O_X\big).$$
The $B$ admits a morphism $q: B\rightarrow\mathbb P(\Omega_X^1(\log D))$ 
 which extends the rational mapping 
$\mathbb P(\Omega^1_X(\log D)\oplus \mathscr O_X)\dashrightarrow\mathbb P(\Omega_X^1(\log D))$ 
associated to  the canonical mapping $\Omega^1_X(\log D)\hookrightarrow\Omega^1_X(\log D)\oplus\mathscr O_X.$ Using the symbol $\mathscr O(1)$ to denote the tautological line sheaf
over  $\mathbb P(\Omega_X^1(\log D))$ and $\mathbb P(\Omega^1_X(\log D)\oplus \mathscr O_X).$
Take a nonzero rational section $s$ of $\Omega^1_X(\log D)$ over $X,$ then $s$ determines a rational 
section $s_1$ of $\mathscr O(1)$ over  $\mathbb P(\Omega_X^1(\log D)).$ The divisor $(s_1)$ is the sum of a generic hyperplane section (on fibers over $X$) and the pull-back of a divisor (on $X$). 
 Notice that $(s,0)$ is  also a nonzero rational section of $\Omega^1_X(\log D)\oplus \mathscr O_X$ over $X,$ hence 
it determines a  rational section 
  $s_2$ of $\mathscr O(1)$ over $\mathbb P(\Omega^1_X(\log D)\oplus \mathscr O_X).$ Note  
  $(s_2)$ is again the sum of a generic hyperplane section and the pull-back of a divisor.
Comparing $q^*(s_1)$ with $p^*(s_2),$ we find that they coincide except that $p^*(s_2)$ contains $[0]$ 
with multiplicity 1. Therefore, we obtain  
\begin{equation}\label{pq}
q^*\mathscr O(1)\cong p^*\mathscr O(1)\otimes \mathscr O(-[0]).
\end{equation}

 Observing the following commutative diagram 
$$
\xymatrix{
S\ar@/_/[ddr]_-{\partial f} \ar@/^/[drr]^-{f'}
\ar[dr]^-f  | (.45){} \\
& X  & \ar[l]\mathbb P(\Omega^1_X(\log D))  \\
& \mathbb P(\Omega^1_X(\log D)\oplus \mathscr O_X)\ar[u] \ar@{-->}[ur] | (.45){} & B,\ar[l]^-p \ar[u]_-q&
}
$$
we see that there is a unique  holomorphic lift $\phi: S\rightarrow B$ satisfying  $f'=q\circ \phi$ and $\partial f=p\circ \phi.$
Combining   (\ref{pq}) with the above diagram, we obtain 
\begin{eqnarray}\label{t}
T_{f',\mathscr O(1)}(r)&=&T_{\phi,q^*\mathscr O(1)}(r)+O(1) \\
&=&
T_{\phi,p^*\mathscr O(1)}(r)-T_{\phi,\mathscr O([0])}(r)+O(1)  \nonumber  \\
&=&T_{\partial f,\mathscr O(1)}(r)-T_{\phi,\mathscr O([0])}(r)+O(1). \nonumber
\end{eqnarray}
Since the global section $(0,1)$ of $\Omega^1_X(\log D)\oplus \mathscr O_X$ over $X$ corresponds to the divisor $[\infty]$ on $\mathbb P(\Omega^1_X(\log D)\oplus \mathscr O_X),$ then 
$\mathscr O([\infty])\cong\mathscr O(1).$ This implies that
$$T_{\partial f,\mathscr O(1)}(r)=m_{\partial f}(r,[\infty])+N_{\partial f}(r,[\infty])+O(1).$$
Indeed, we notice  that $\partial f$ meets $[\infty]$ only over $D,$  with multiplicity at most 1, 
then $N_{\partial f}(r,[\infty])\leq N_f^{[1]}(r,D).$ Thus, it follows from (\ref{t}) that 
\begin{eqnarray*}
T_{f',\mathscr O(1)}(r)
&\leq& m_{\partial f}(r,[\infty])+N_f^{[1]}(r,D)
-m_\phi(r,[0])-N_\phi(r,[0])+O(1).
\end{eqnarray*}
In the above inequality, $m_\phi(r,[0])$ is bounded from below; $N_\phi(r,[0])$ is 
equal to $N_{f,{\rm{Ram}} (D)}(r);$ and $m_{\partial f}(r,[\infty])$ is bounded from above by 
$O(\log T_{f,\mathscr A}(r)-\kappa(r)r^2+\log^+\log r),$   more precisely, bounded from above by 
$O(\log T_{f,\mathscr A}(r)+r)$  if $S$ is the Poincar\'e disc, due to Lemma \ref{key}. This proves the theorem.

\begin{theorem}\label{thm}  Let  $X$ be a smooth complex projective curve, and let $D$ be an effective $($reduced$)$ divisor on $X.$ Then 
for  any  holomorphic curve  $f:S\rightarrow X$ which is nonconstant, we have
 \begin{eqnarray*}
T_{f, \mathscr K_X(D)}(r) &\leq_{\rm exc}&
N^{[1]}_f(r,D)+
O\Big(\log T_{f, \mathscr A}(r)-\kappa(r)r^2+\log^+\log r\Big),
\end{eqnarray*}
where $\kappa$ is defined by $(\ref{kappa}).$ More precisely, if $S$ is the Poincar\'e disc $($take $o$  as the center of disc$)$,  then
 \begin{eqnarray*}
T_{f, \mathscr K_X(D)}(r) &\leq_{\rm exc}&
N^{[1]}_f(r,D)+
O\Big(\log T_{f, \mathscr A}(r)+r\Big).
\end{eqnarray*}
\end{theorem}
\begin{proof} Since $X$ is a  curve, then $\Omega_X^1(\log D)$ is isomorphic to $\mathscr K_X(D)$ and thus 
the canonical projection $\pi:\mathbb P(\Omega_X(\log D))\rightarrow X$ is  an isomorphism. It yields that $\mathscr O(1)\cong\pi^*\mathscr K_X(D)$ and $f'=\pi^{-1}\circ f.$ Therefore,  we have
$$T_{f',\mathscr O(1)}(r)=T_{f, \mathscr K_X(D)}+O(1).$$
By Theorem \ref{tau}, we have the theorem holds.
\end{proof}

\begin{cor}  Let  $X$ be a smooth complex projective curve with $\mathscr K_X$ ample. 
Then there exist  positive constants $a, b$ such that 
 \begin{eqnarray*}
T_{f, \mathscr K_X}(r) &\leq_{\rm exc}& 
a-b\kappa(r)r^2,
\end{eqnarray*}
where $\kappa$ is defined by $(\ref{kappa}).$ More precisely, if $S$ is the Poincar\'e disc $($take $o$  as the center of disc$)$,  then
 \begin{eqnarray*}
T_{f, \mathscr K_X}(r) &\leq_{\rm exc}& 
a+br.
\end{eqnarray*}
\end{cor}

\begin{proof} According  to the first conclusion of Theorem \ref{thm}, we have
 \begin{eqnarray*}
T_{f, \mathscr K_X}(r) &\leq_{\rm exc}&
O\Big(\log T_{f, \mathscr K_X}(r)-\kappa(r)r^2+\log^+\log r\Big).
\end{eqnarray*}
If $\kappa(r)\equiv0,$ then the universal covering of $S$ is  $\mathbb C.$ By lifting $f$ to the universal covering, one can assume that $S=\mathbb C.$ 
 Since $\mathscr K_X>0,$ then we note from \cite{Noguchi} that     
$T_{f, \mathscr K_X}(r)\geq_{\rm exc} O(\log r)$ if $f$ is nonconstant. 
Using the above inequality, $f$  has to be a constant. 
 If $\kappa(r)\not\equiv0,$  we see that 
$\log^+\log r\leq_{\rm exc} -\kappa(r)r^2$ since $-\kappa(r)\geq0$ is increasing. 
Again, by  
$\log T_{f, \mathscr K_X}(r)=o(T_{f, \mathscr K_X}(r)),$  one obtains  $T_{f, \mathscr K_X}(r)\leq_{\rm exc} O(-\kappa(r)r^2).$ 
So, we show the first conclusion of the corollary. 
The second conclusion of the corollary can be 
  proved similarly by 
   using the second conclusion  of Theorem \ref{thm}. 
\end{proof}

\section{Second Main Theorem}

\subsection{Curvature current inequality}~

In order to prove the main theorem (Theorem \ref{main}), it remains  to introduce a curvature current inequality obtained by Sun (\cite{sun}, Proposition 2.4).
Since this  inequality was demonstrated, then we are not going to give details again. 
However, in order to make it more readable for  the reader to  understand the hypotheses 
 of Sun's curvature current inequality,
  we shall provide  
  a concise explanation,  but  
  the details  please refer to [\cite{sun}, Section 2], [\cite{vz1}, Section 4] and [\cite{vz2}, Section 6]. 

\subsubsection{Higgs bundles and Hodge bundles}~

Higgs bundles were introduced by Hitchin  \cite{Hit} as solutions of the so-called Hitchin equations, the 2-dimensional reduction of the Yang-Mills self-duality equations, given by
$$F_A+[\tau, \tau^*]=0, \ \ \  \bar\partial_A\tau=0,$$
where $F_A$ is the curvature of a unitary connection $\nabla_A=\partial_A+\bar\partial_A$ associated to a Dolbeault operator $\bar\partial_A$ on a holomorphic principal $G_{\mathbb C}$-bundle $F.$ Exactly, 
a Higgs bundle over a smooth projective variety $X$ is  a pair $(F, \tau)$ such that

$a)$  $F$ is a holomorphic vector bundle over $X;$

$b)$  $\tau$ is a holomorphic 1-form with values in the bundle of endomorphisms of $F,$ satisfying 
$\tau\wedge\tau=0.$ We call $\tau$ the Higgs field.

To learn Viehweg-Zuo's construction of  certain Higgs bundles, we refer the reader to [\cite{vz1}, Section 4] and [\cite{vz2}, Section 4]. Viehweg-Zuo's construction gives a method in constructing (pseudo) Finsler metrics. By which, Sun \cite{sun} obtained a
curvature current inequality (Lemma \ref{cci} below).

Let $\mathcal M_g$ 
 be the moduli space of algebraic curves of genus $g$  over a scheme. A
  Hodge bundle over $\mathcal M_g$
 is a vector bundle  $E$  
 whose fiber at a point $\mathcal C\in \mathcal M_g$ 
 is the space of holomorphic differentials on the curve $\mathcal C$. To be precise, see \cite{HM}, let $\pi: \mathfrak C_g\rightarrow \mathcal M_g$
 be the universal algebraic curve of genus $g,$ and  let $\mathscr E$
 be its relative dualizing sheaf, the Hodge bundle $E$ is the pushforward of $\mathscr E$, i.e., 
 $$E=\pi_*\mathscr E.$$
 
\subsubsection{Curvature current inequality}~

Now, we introduce an inequality of curvature currents proved  by Sun  \cite{sun}, who showed his  inequality via  Viehweg-Zuo's construction \cite{vz1,vz2}.  Without going into the details about   moduli spaces, we refer the reader to  \cite{A-R, popp}.

Let $(X,D)$ be a smooth logarithmic pair over $\mathbb C.$   Put $\hat S=S\setminus f^*D,$  where $f: S\rightarrow X$ is a holomorphic curve  whose image is not  contained in ${\rm{Supp}}D.$ 
Let $\gamma: \hat S\rightarrow X\setminus D$ be the restriction of $f$ to $\hat S,$
which induces a  natural lift 
$$\gamma':\hat S\rightarrow \mathbb P(T_X(-\log D)),$$
where 
$$\mathbb P(T_X(-\log D)):={\rm{\textbf{Proj}}} \ {\rm{\textbf{Sym}}}^\bullet\Omega^1_X(\log D)$$ 
is the projective logarithmic tangent bundle.
 Suppose  $(\psi: V\rightarrow U:=X \setminus D)$ is  a smooth family of polarized smooth varieties with semi-ample canonical sheaves and a given Hilbert polynomial $h,$ such that the induced classifying mapping from $U$ to the  moduli scheme $\mathcal M_h$ is quasi-finite. 
Follow the theory of Viehweg-Zuo (\cite{vz1}, Section 4; \cite{vz2}, Section 6), 
we  shall  have the following
geometric objects over $X$ (\cite{sun}, Section 2):  
an ample line bundle $A$ whose restriction 
on  smooth locus $X\setminus D$ is isomorphic to some Viehweg line bundle
$\det(\psi_*\omega^\nu_{V/U})^\mu$ (\cite{vz2}, Corollary 2.4 (ix) and  Section 4); 
a deformation Higgs bundle $(F,\tau)$
associated to   $\psi$; a logarithmic Hodge bundle $(E, \theta)$ with poles along $D+T,$
 where $T$ is a normal crossing divisor; 
and  a comparison mapping $\rho$ which 
  fits  into the following commutative diagram
   (\cite{vz1}, Lemma 4.4; \cite{vz2}, Lemma 6.3)
$$
		\centerline{\xymatrix{F^{p,q}\ar[r]^-{\tau^{p,q}}\ar[d]^-{\rho^{p,q}} & F^{p-1,q+1}\otimes\Omega_X^1(\log D)\ar[d]^-{\rho^{p-1,q+1}\otimes\iota} \\
				A^{-1}\otimes E^{p,q}\ar[r]^-{{\rm{Id}}\otimes\theta^{p,q}}&A^{-1}\otimes E^{p-1,q-1}\otimes\Omega_X^1(\log(D+T))}}
		$$ 
	where $\mathscr O_X$ is a subsheaf of  $F^{d,0}$ and that  (\cite{vz2}, Lemma 6.5)
	$$(E, \theta)=\bigoplus_{p, q} \big(E^{p,q}, \theta^{p,q}\big), \ \ \   (F, \tau)=\bigoplus_{p, q} \big(F^{p,q}, \tau^{p,q}\big).$$	
\ \ \ \	   To avoid excessive  complexity and verbosity of the statements (since that is not the  point of the paper), the reader may refer to  [\cite{vz2}, Pages 20-21] for the  definitions  of $F^{p,q}, E^{p,q}, \tau^{p,q}, \theta^{p,q},$ and  [\cite{vz2}, Section 4] for the concept of determinant bundle 
	$\det(\psi_*\omega^\nu_{V/U})$.
Let  $d$ be the fiber dimension of  family  $\psi.$ 
By iterating the Higgs mappings  
 $\tau^{p,q},$ it follows that  (\cite{vz1}, Lemma 4.4 (vi))
   	$$\tau^{d-q+1, q-1}\circ\cdots\circ\tau^{d,0}:   F^{d,0}\rightarrow F^{d-q,q}\otimes \bigotimes^q\Omega_X^1(\log D).$$		
Since $\tau\wedge\tau=0$ (Section 5.1, Part A), then the composition factors through (\cite{vz1}, Lemma 4.4 (vi))
  	$$\tau^q:F^{d,0}\longrightarrow F^{d-q,q}\otimes {\rm{\textbf{Sym}}}^q\Omega_X^1(\log D).$$		
The pull-back of the Higgs bundle $(F,\tau)$ over $X$ by $\gamma$ induces a Higgs bundle $(F_{\gamma},\tau_\gamma)$ over $\hat S$ in the following manner
$$F_\gamma:=\gamma^*F, \ \ \  \tau_\gamma: \xymatrix{F_\gamma\ar[r]^-{\gamma^*\tau}&F_\gamma\otimes
\gamma^*\Omega_X^1(\log D)\ar[r]&F_\gamma\otimes\Omega_{\hat S}^1.
}
$$
Similarly as above, we  define the holomorphic Hodge bundle $(E_\gamma,\theta_\gamma)$, where $\theta_\gamma$ has logarithmic pole along $\gamma^*T.$
Over $\hat S,$ we  have the  
commutative diagram (\cite{vz1}, Lemma 4.4 (i))
$$
		\centerline{\xymatrix{F^{p,q}_\gamma\ar[r]^-{\tau^{p,q}_\gamma}\ar[d]^-{\rho^{p,q}_\gamma} & F^{p-1,q+1}_\gamma\otimes\Omega_{\hat S}^1\ar[d]^-{\rho^{p-1,q+1}_\gamma\otimes\iota} \\
				A^{-1}_\gamma\otimes E^{p,q}_\gamma\ar[r]^-{{\rm{Id}}\otimes\theta^{p,q}_\gamma}&A^{-1}_\gamma\otimes E^{p-1,q-1}_\gamma\otimes\Omega_{\hat S}^1(\log\gamma^*T),}}
		$$
where $A_\gamma:=\gamma^*A.$ Similarly, we  can define the iterations of Higgs mappings 
$$\tau_\gamma^q: T_{\hat S}^{\otimes q}\longrightarrow F_\gamma^{d-q,q}.$$
By using the iterations of Higgs mappings,  we  shall define a Higgs subbundle of   $(E_\gamma,\theta_\gamma)$. For each integer $q\geq0,$ 
 we define $G^{d-q,q}$  as the \emph{saturation} of the image of 
$$\xymatrix{A\otimes T_{\hat S}^{\otimes q}\ar[r]^-{{\rm{Id}}\otimes\tau^q_\gamma}&A\otimes F_\gamma^{d-q,q}\ar[r]&E_\gamma^{d-q,q}}
$$
in $E_\gamma^{d-q,q}.$ We have $\theta_\gamma^{d-q,q}(G^{d-q,q})\subseteq G^{d-q-1,q+1}\otimes\Omega^1_{\hat S}$ (\cite{sun}, Lemma 2.1) and 
$c_1(\det G, h)\leq0$ (\cite{sun}, Proposition 2.2), where $h$ is the Hermitian metric on the determinant bundle $\det G$  induced 
by the Hodge metric on $E.$
Let $\mathscr O(-1)$ be the tautological line bundle over $\mathbb P(T_X(-\log D)).$ 
Since the  iterations of Higgs mappings $\tau_\gamma^q$  also factors through
$$
\xymatrix{T_{\hat S}^{\otimes q}\ar[r]&\gamma'^*\mathscr O(-q)\ar[r]^-{\gamma'^*\tilde\tau^q}&F_\gamma^{d-q,q}}, 
$$
in which  $\tilde \tau^q:\mathscr O(-q)\rightarrow\pi^*F^{d-q,q}$ is the lift of $\tau^q$ and $\pi:\mathbb P(T_X(-\log D))\rightarrow X$ is the nature projection,  then  $G^{d-q,q}$ is also the saturation of the image of
$$\xymatrix{A_\gamma\otimes \gamma'^*\mathscr O(-q)\ar[r]^-{{\rm{Id}}\otimes\gamma'^*\tilde\tau^q}&
A_\gamma\otimes F_\gamma^{d-q,q}\ar[r]&E_\gamma^{d-q,q}},
$$
 where $\pi:\mathbb P(T_X(-\log D))\rightarrow X$ is the nature projection.
Then this gives the mappings (\cite{sun}, (2.4))
$$\zeta^q: \gamma'^*\mathscr O(-q)\longrightarrow A_\gamma^{-1}\otimes G^{d-q,q}, \ \ \ q=0,\cdots,d.$$
Note that $\rho_\gamma^{d-1,1}\circ\tau^1_\gamma$ is nonzero, then it implies that $\zeta^1$ is nonzero. So, there exists a positive integer $m$ such that $\zeta^m\not=0$ and $\zeta^{m+1}=0,$  here $m$ is called the maximal length of iteration. We have $m\leq d$ and 
$\det G=\bigotimes_{q=0}^mG^{d-q,q}$ ($G^{d-q,q}$=0 for $q>m$). 
 For every $q,$ one can construct a (pseudo) metric $F_q$ on $\mathscr O(-1)$ via  the following  composition mapping
$$\xymatrix{\mathscr O(-q)\ar[r]^-{\tilde\tau^q}&\pi^*F^{d-q,q}\ar[r]&\pi^*(A^{-1}\otimes E^{d-q,q})}.$$
Here, we note that $F_q$ is a bounded pseudo metric with possible degeneration on $\mathbb P(T_X(-\log D)).$ So, we can write $F_q=\phi_q F,$ where $F$ is a smooth metric on $\mathscr O(-1),$ and $\phi_q$ is a bounded function with  at most a  discrete set of  zeros. Then we have
$$c_1(\mathscr O(-1), F)=c_1(\mathscr O(-1), F_q)+dd^c\log\phi_q.$$
This will  derive  a curvature current inequality as follows 
\begin{lemma}[\cite{sun}, Proposition 2.4]\label{cci} Suppose  that $m$ is the maximal length of iteration, i.e., the largest integer such that $\zeta^m\not=0.$ 
Then
$$\frac{m+1}{2}f'^*c_1(\mathscr O(-1))\leq -f^*c_1(A)+\frac{1}{m}\sum_{q=1}^mqf'^*dd^c\log\phi_q,$$
where $f, A, \phi_q$ are given as above.
\end{lemma}

\subsection{Second Main Theorem}~

We show the following Second Main  Theorem 
\begin{theorem}[Second Main Theorem]\label{main} 
Let $(X, D)$ be a smooth logarithmic pair over $\mathbb C$ with $U= X \setminus D.$ Assume that there is a smooth family $(\psi: V\rightarrow U)$ of polarized smooth varieties with semi-ample canonical sheaves and a given Hilbert polynomial $h,$ such that the induced classifying mapping from $U$ into  moduli scheme $\mathcal M_h$ is quasi-finite. 
Then for any holomorphic curve  $f:S\rightarrow X$ whose image is not contained in ${\rm{Supp}}D,$
we have
\begin{eqnarray*}
T_{f,A}(r) &\leq_{\rm exc}&
\frac{d+1}{2}N^{[1]}_f(r,D)+
O\Big(\log T_{f, A}(r)-\kappa(r)r^2+\log^+\log r\Big),
\end{eqnarray*}
where $A$ is an ample line bundle over $X$ given in Lemma $\ref{cci},$   $d$ is the fiber dimension of the family $\psi,$ and $\kappa$ is defined by $(\ref{kappa}).$  More precisely, if $S$ is the Poincar\'e disc $($take $o$  as the center of disc$)$,  then
\begin{eqnarray*}
T_{f,A}(r) &\leq_{\rm exc}&
\frac{d+1}{2}N^{[1]}_f(r,D)+
O\Big(\log T_{f, A}(r)+ r\Big).
\end{eqnarray*}
\end{theorem} 
\begin{proof} By lifting $f$ to the (analytic) universal covering, we may assume that $S$ is simply connected. 
By Lemma \ref{cci}, it  follows immediately that 
\begin{eqnarray*}
-\frac{d+1}{2}T_{f',\mathscr O(1)}(r)&\leq&-\frac{m+1}{2}T_{f',\mathscr O(1)}(r) \\
&\leq& -T_{f,A}(r)+\frac{\pi}{m}
\sum_{q=1}^mq\int_{D(r)}g_r(o,x)dd^c\log\phi_q\circ f'.
\end{eqnarray*}
Since $\phi_1,\cdots, \phi_q$ are bounded, then it yields  from Coarea formula and Dynkin formula that 
\begin{eqnarray*}
\pi\int_{D(r)}g_r(o,x)dd^c\log\phi_q\circ f' &=&
\frac{1}{4}\mathbb E_o\left[\int_0^{\tau_r}\Delta_S\log\phi_q\circ f'(X_t)dt\right]\\
&=& \frac{1}{2}\mathbb E_o\left[\log\phi_q\circ f'(X_{\tau_r})\right]+O(1) \\
&\leq& O(1).
\end{eqnarray*}
Combining the above  with Theorem \ref{tau},  we can prove the theorem.
\end{proof}

\begin{cor}\label{cor3}  Assume the same conditions as in  Theorem $\ref{main}$.
Let $f:S\rightarrow X$ be a  holomorphic curve  ramifying over $D$ with multiplicity $c>(d+1)/2.$  If $f$ satisfies the growth condition
$$\liminf_{r\rightarrow\infty}\frac{\kappa(r)r^2}{T_{f,A}(r)}=0,$$
where $\kappa$ is defined by $(\ref{kappa}),$ then $f(S)$ 
is contained in $D.$  More precisely, if $S$ is the Poincar\'e disc $($take $o$  as the  center of disc$)$,  then $f(S)$ 
is contained in $D$ provided that  
$$\limsup_{r\rightarrow\infty}\frac{r}{T_{f,A}(r)}=0.$$
\end{cor} 
\begin{proof}  Set $b=2c/(d+1),$ then $b>1.$ By contradiction,  we assume that $f(S)$ 
is not contained in $D.$
By the first conclusion of Theorem \ref{main} and condition $f^*D\geq c\cdot{\rm {Supp}}f^*D,$ we obtain
\begin{eqnarray*}
bT_{f,A}(r) &\leq_{\rm exc}&
cN^{[1]}_f(r,D)+
O\Big(\log T_{f, A}(r)-\kappa(r)r^2+\log^+\log r\Big) \\
&\leq& N_f(r,D)+
O\Big(\log T_{f, A}(r)-\kappa(r)r^2+\log^+\log r\Big). 
\end{eqnarray*}
If $\kappa(r)\not\equiv0,$  by First Main Theorem, then  the above inequality  implies   that  $b\leq 1,$ it is a contradiction.
If $\kappa(r)\equiv0,$
 we can regard $S$ as $\mathbb C$ by lifting $f$ to the universal covering, then we obtain $T_{f,A}(r)\geq_{\rm exc} O(\log r)$ (see \cite{Noguchi}). 
But, it still contradicts with $b>1.$  Therefore, the first conclusion of the corollary holds.  The second conclusion of the corollary is similarly proved.
\end{proof}

\begin{cor}\label{cor1} Assume the same conditions as in Theorem $\ref{main}$. Then 
for any holomorphic curve  $f:S\rightarrow X$ whose image is not contained in ${\rm{Supp}}D,$
we have 
\begin{eqnarray*}
T_{f,K_X(D)}(r) &\leq_{\rm exc}&
\frac{k(d+1)}{2}N^{[1]}_f(r,D)+
O\Big(\log T_{f, A}(r)-\kappa(r)r^2+\log^+\log r\Big)
\end{eqnarray*}
 for an   integer $k$ such that  $A^{\otimes k}\geq K_X(D),$ where $\kappa$ is defined by $(\ref{kappa}).$ More precisely, if $S$ is the Poincar\'e disc $($take $o$  as the center of disc$)$,  then
 \begin{eqnarray*}
T_{f,K_X(D)}(r) &\leq_{\rm exc}&
\frac{k(d+1)}{2}N^{[1]}_f(r,D)+
O\Big(\log T_{f, A}(r)+r\Big)
\end{eqnarray*}
 for an   integer $k$ such that  $A^{\otimes k}\geq K_X(D).$ 
\end{cor}
\begin{proof} 
Since $A^{\otimes k}\geq K_X(D),$ then 
we see that    $T_{f,K_X(D)}(r)\leq k T_{f,A}(r)+O(1).$   Therefore,  the corollary follows from 
Theorem \ref{main}.
\end{proof}

Let us consider  Siegel modular varieties \cite{Sig}. 
A Siegel modular variety  is a moduli space of principally polarized Abelian varieties of a fixed dimension.  Exactly speaking, the Siegel modular variety  
$\mathcal A_g$  parametrizes the principally polarized Abelian varieties of dimension $g$, which can be constructed as the complex analytic spaces (constructed as the quotient of  the Siegel upper half-space of degree $g$ by the action of a symplectic group). 
Refer  to \cite{H-S, Tai},   
 $\mathcal A_g$ has dimension $g(g+1)/2,$  and is of general type for $g\geq7.$ 
A Siegel modular variety $\mathcal A_g^{[n]},$ which 
parametrizes the principally polarized Abelian varieties of dimension $g$ with  level-$n$ structure, 
 arises as the quotient of   Siegel upper half-space by the action of the principal congruence subgroup of level-$n$ of a symplectic group.

\begin{theorem}\label{mm} Let $\mathcal A_g^{[n]}$ $(n\geq 3)$ be the moduli space of principally polarized Abelian 
varieties with level-$n$  structure. 
Let $\overline{\mathcal A}_g^{[n]}$ be the  smooth compactification of $\mathcal A_g^{[n]}$ such that  $D=\overline{\mathcal A}_g^{[n]} \setminus\mathcal A_g^{[n]}$
is a normal crossing $($boundary$)$ divisor.  
For  any holomorphic curve $f:S\rightarrow \overline{\mathcal A}_g^{[n]}$  whose image is  not  contained in ${\rm{Supp}}D,$ we have  
\begin{eqnarray*}
T_{f,K_{\overline{\mathcal A}_g^{[n]}}(D)}(r) &\leq_{\rm exc}&
\frac{(g+1)^2}{2}N^{[1]}_f(r,D)+
O\Big(\log T_{f, A}(r)-\kappa(r)r^2+\log^+\log r\Big),
\end{eqnarray*}
where  $\kappa$ is defined by $(\ref{kappa}).$ More precisely, if $S$ is the Poincar\'e disc $($take $o$  as the center of disc$)$,  then
\begin{eqnarray*}
T_{f,K_{\overline{\mathcal A}_g^{[n]}}(D)}(r) &\leq_{\rm exc}&
\frac{(g+1)^2}{2}N^{[1]}_f(r,D)+
O\Big(\log T_{f, A}(r)+r\Big).
\end{eqnarray*}
\end{theorem}
\begin{proof}
We just need to carry the arguments of Sun (\cite{sun}, Corollary 4.3) and use Corollary \ref{cor1}, 
 then  the  theorem can be  proved.
\end{proof}

\noindent\textbf{Acknowledgement.} The author is  very grateful to   
Prof. Songyan Xie and Dr. Ruiran Sun for  the valuable  discussions with them, and grateful  to  Dr. Yan He for his useful suggestions. 
The author also thanks  referee for his/her valuable comments and suggestions. 
\vskip\baselineskip

\label{lastpage-01}

\vskip\baselineskip
\vskip\baselineskip

\end{document}